\documentclass[a4paper,10pt]{article}
\usepackage[utf8]{inputenc}

\usepackage{wrapfig}
\usepackage{amsmath} % numeraci— de teoremes subetidata a les seccions, per exemple
\usepackage{amsthm} % Altres imprescindibles, com environement proof.
\usepackage{amssymb} % carˆcters matemˆtics, imprescindible
\usepackage{enumerate} % enumerar amb coses diferents a nombres
\usepackage{esint} %integrals
\usepackage{pgf,tikz} %tikz: dibuixos i esquemes, com el cercle
\usetikzlibrary{arrows} %fletxes al tikz
 \usepackage{yfonts} % textgoth, textswab, textfrak
 \usepackage{mathrsfs} % mathscr
 \usepackage{mathabx} %antitransformada fourier
 \usepackage{graphicx}
\usepackage{caption}
\usepackage{subfigure}
\usepackage{mathtools} %alinear matrius

\definecolor{ffffff}{rgb}{1.0,1.0,1.0}
\definecolor{qqqqff}{rgb}{0.0,0.0,1.0}
\definecolor{ffqqqq}{rgb}{1.0,0.0,0.0}
\definecolor{zzzzqq}{rgb}{0.6,0.6,0.0}
\definecolor{marronet}{rgb}{0.6,0.2,0}
\definecolor{negre}{rgb}{0,0,0}
\definecolor{vermell}{rgb}{0.8,0.05,0.05}
\definecolor{blau}{rgb}{0.2,0.1,1}
\definecolor{blauclar}{rgb}{0.,0.,1.}
\definecolor{grisfosc}{rgb}{0.25098039215686274,0.25098039215686274,0.25098039215686274}
\definecolor{verd}{rgb}{0.1,0.6,0.1}
\definecolor{taronja}{rgb}{0.9,0.6,0.05}
\definecolor{vermellclar}{rgb}{1.,0.,0.}
\definecolor{verdet}{rgb}{0,0.8,0.1}
\definecolor{blauverd}{rgb}{0,0.4,0.2}
\definecolor{grisclar}{rgb}{0.6274509803921569,0.6274509803921569,0.6274509803921569}

\newcommand*\circled[1]{\tikz[baseline=(char.base)]{
            \node[shape=circle,draw,inner sep=2pt] (char) {#1};}}

\newcommand{\C}{{\mathbb C}}       % Field of complex numbers
\newcommand{\R}{{\mathbb R}}       % Field of real numbers
\newcommand{\N}{{\mathbb N}}       % Natural numbers
\newcommand{\DDD}{{\mathbb D}}
\newcommand{\rf}[1]{{(\ref{#1})}}
\newcommand{\supp}{{\rm supp}}

\newcommand{\Beurling}{{\mathcal B}}
\newcommand{\Cauchy}{{\mathcal C}}
\newcommand{\norm}[1]{{\left\| {#1} \right\|}}
\newtheorem{theorem}{Theorem}%[Section]
\newtheorem*{theorem*}{Theorem}%[Section]
\newtheorem{lemma}[theorem]{Lemma}

\newtheorem{corollary}[theorem]{Corollary}
\newtheorem*{corollary*}{Corollary}

\newtheorem{conjecture}[theorem]{Conjecture}
\newtheorem{definition}[theorem]{Definition}

\newtheorem{remark}[theorem]{Remark}

\numberwithin{subsection}{section}
\numberwithin{theorem}{section}
\numberwithin{equation}{section}
\numberwithin{figure}{section}

\usepackage[affil-it]{authblk}

\title{Beltrami equations in the plane and Sobolev regularity}

\author{Mart\'i Prats
\thanks{MP (De\-par\-ta\-men\-to de Ma\-te\-m\'a\-ti\-cas, U\-ni\-ver\-si\-dad Au\-t\'o\-no\-ma de Ma\-drid - ICMAT, Spain): \texttt{marti.prats@uam.es}.}}

\begin{document}
\maketitle
\bibliographystyle{alpha}

\begin{abstract} 
New results regarding the Sobolev regularity of the principal solution of the linear Beltrami equation $\bar{\partial} f = \mu \partial f + \nu \overline{\partial f}$ for discontinuous Beltrami coefficients $\mu$ and $\nu$ are obtained, using Kato-Ponce commutators, obtaining that the $\overline \partial f$ belongs to a Sobolev space with the same smoothness as the coefficients but some loss in the integrability parameter. A conjecture on the cases where the limitations of the method do not work  is raised.
\end{abstract}

\section{Introduction}
Given $\mu, \nu \in L^\infty$ compactly supported satisfying the elliptic condition
\begin{equation}\label{eqElliptic}
\norm{|\mu|+|\nu|}_{L^\infty}=\kappa<1,
\end{equation}
the Beltrami equation 
\begin{equation}\label{eqBeltrami}
\bar{\partial} f = \mu \partial f + \nu \overline{\partial f}
\end{equation}
has a unique homeomorphic solution $f\in W^{1,2}_{loc}$ such that $f(z)-z=\mathcal{O}_{z\to\infty}(1/z)$, which we call principal solution to \rf{eqBeltrami}. The existence of this solution depends deeply on the fact that the Beurling transform 
$$\Beurling f = - p.v.  \frac1{\pi z^2} * f$$
is bounded on $L^p$ spaces for $1<p<\infty$ and unitary in $L^2$ (see \cite[Chapter 4]{AstalaIwaniecMartin}, for instance). 

For each $\kappa<1$, we define the extremal exponent $p_\kappa:=1+\frac{1}{\kappa}>2$. 
In 1992, Kari Astala published a celebrated theorem on the area distortion of quasiconformal mappings, which implies that every quasiconformal mapping $f$ with Beltrami coefficient $\mu\in L^\infty_{c}$ such that $\norm{\mu}_{L^\infty}=\kappa<1$ satisfies that
\begin{equation}\label{eqAstala}
\bar{\partial} f \in L^p \mbox{ \quad\quad whenever } \frac{1}{p_\kappa} < \frac1p \leq 1 
\end{equation}
(see \cite[Corollary 1.2]{Astala}). Some years later, Astala, Iwaniec and Saksman found the following remarkable (and sharp) result.
\begin{theorem}[{\cite[Theorem 3]{AstalaIwaniecSaksman}}]\label{theoAIS}
Given $\mu, \nu \in L^\infty$ with $\norm{|\mu|+|\nu|}_{L^\infty}=\kappa<1$, the operator
$$Id-\mu\Beurling-\nu\overline{\Beurling}$$
is invertible on $L^p(\C)$ for $\frac{1}{p_\kappa}<\frac1p<\frac{1}{p_\kappa'}$, with
\begin{equation}\label{eqInverseBeltramiOperatorUniformControl}
\norm{(Id-\mu\Beurling-\nu\overline{\Beurling})^{-1}}_{L^p\to L^p}\leq C_{\kappa,p}.
\end{equation}
\end{theorem}
When the coefficients satisfy some extra assumption, we can improve the previous results. We discuss this results in Section \ref{secPreviousResults}, after a brief introduction of the function spaces we consider. Nevertheless, we sketch here the general facts.  Given $0<s<\infty$ and $1<p<\infty$, we say that the Sobolev space $W^{s,p}(\R^{d})$ (in the sense of Bessel potential spaces, see \cite[Section 2.2.2]{TriebelTheory}) is critical if $s-\frac dp=0$. If $s-\frac dp>0$ it is called supercritical and if, instead, $s-\frac dp<0$, then it is called subcritical. Note that so-called differential dimension $s-\frac dp$ coincides with the homogeneity exponent of the semi-norms of these spaces. The functions in supercritical spaces are continuous, while the functions on critical spaces  are just in the space of vanishing mean oscillation functions ($VMO$), i.e., the closure of $C^\infty_c$ in $BMO$. For the subcritical spaces we have less self-improvement (see Section \ref{secSpaces} below). 

Roughly speaking, when the coefficients $\mu$ and $\nu$ are in a supercritical Sobolev space, then $\bar\partial f$ inherits the regularity of the Beltrami coefficient. In the critical situation, there is a small loss, and in the subcritical case, there is a bigger gap, in the spirit of \rf{eqAstala}, where $\mu$ is in every $L^p$ space but $\bar\partial f$ is only $p$-integrable for a certain range.

The supercritical case is well understood (see \cite[Chapter 15]{AstalaIwaniecMartin} for H\"older spaces and \cite{ClopFaracoMateuOrobitgZhong} and \cite{CruzMateuOrobitg} for Sobolev, as well as Besov and Triebel-Lizorkin spaces). The critical case is studied in the latter two papers as well as in \cite{BaisonClopOrobitg} and \cite{BaisonClopGiovaOrobitgNapoli}, while the literature on the subcritical cases is less complete (see \cite{ClopFaracoMateuOrobitgZhong} and \cite{ClopFaracoRuiz}). This note is devoted to unify the approaches for the critical and subcritical situations, in the quest to find a complete sharp theory. 

Sharp bounds in this theory may lead to a better understanding of the stability of the Calder\'on inverse problem, as shown in \cite{ClopFaracoRuiz}. There, the authors prove that if one knows all the possible pairs of Dirichlet and  Neumann data of the solutions to the conductivity equation for conductivities satisfying certain a priori subcritical Sobolev conditions, then the recovery is stable. A crucial step there is to solve a Beltrami equation as \rf{eqBeltrami} above: after showing Sobolev regularity of the principal solution to a certain family of equations, the authors show an asymptotic decay of the so-called Complex Geometric Optics Solution. Greater Sobolev regularity of these solutions is translated into higher decay of the solutions and better stability estimates.

 We show the following result.
\begin{theorem}\label{theoRegularityCoefficients}
Let $0<s<2$, $1<p<\infty$, let $\mu,\nu \in W^{s,p}_{c}(\C)\cap L^\infty$ satisfy \rf{eqElliptic} for $\kappa<1$ and let $f$ be the principal solution to the Beltrami equation \rf{eqBeltrami}. 

If $s=\frac2p$, then
 \begin{equation}\label{eqBigTarget2}
 \bar{\partial} f\in W^{s,q} \mbox{ \quad\quad for every } \frac{1}{q} > \frac{1}{p}.
 \end{equation}

If $s<\frac2p$ and  $\frac{1}{p}<\frac{1}{p_\kappa'}-\frac{1}{p_\kappa} = \frac{1-\kappa}{1+\kappa}$, then
 \begin{equation}\label{eqBigTarget}
 \bar{\partial} f\in W^{s,q} \mbox{ \quad\quad for every } \frac{1}{q} > \frac{1}{p}+\frac{1}{p_\kappa}.
 \end{equation}
\end{theorem}

However, the restriction $\frac{1}{p}<\frac{1}{p_\kappa'}-\frac{1}{p_\kappa}$ seems rather unnatural (see Section \ref{secPreviousResults}). Due to this fact, there is only room for $p$ in the conditions for \rf{eqBigTarget} if $s<2\frac{1-\kappa}{1+\kappa}$, which is equivalent to $\kappa<\frac{2-s}{2+s}$. Therefore it is natural to ask whether it can be removed or not (see Conjecture \ref{conjRegularityCoefficients} below).

Using some embeddings explained in Section \ref{secKnown}, we can deduce the following corollary, which covers the case when $\frac{1}{p}\geq \frac{1}{p_\kappa'}-\frac{1}{p_\kappa}$ as well.
\begin{corollary}\label{coroRegularityCoefficients}
Let $0<s<2$, $1<p<\infty$, with $s<\frac2p$. If $0< \Theta\leq 1$ with $\frac{\Theta}{p}<\frac{1}{p_\kappa'}-\frac{1}{p_\kappa}$, then
 $$\bar{\partial} f\in W^{\theta s, q} \mbox{ \quad\quad for every } \frac{1}{q} > \frac{\Theta}{p}+\frac{1}{p_\kappa}.$$
\end{corollary}

The paper is organized in the following way. Section \ref{secKnown} is devoted to making the background of this article clear. In Section \ref{secSpaces} the definitions and basic properties of the Triebel-Lizorkin and related spaces are given. Section \ref{secCompactlySupported}  specifies some properties of compactly supported Triebel-Lizorkin functions, in order to make a clear picture of the problem and to provide the reader a guide to understand the full scale of Sobolev regularity obtained for the principal mappings in Theorem \ref{theoRegularityCoefficients}. In Section \ref{secPreviousResults} there is a discussion on the existing results using the concepts introduced in the former sections. Finally, Section \ref{secProof} contains the proof of Theorem \ref{theoRegularityCoefficients}.

\section{Background}\label{secKnown}

\subsection{Definitions and well-known properties of function spaces}\label{secSpaces}
First we recall some results on Triebel-Lizorkin spaces.

Let $\{\psi_j\}_{j=0}^\infty\subset C^\infty_c(\R^d)$ with $\psi_0$  supported in $\DDD(0,2)$, $\psi_j$ supported in $\DDD(0,2^{j+1})\setminus \DDD(0,2^{j-1})$ for  $j\geq 1$,  such that $\sum_{j=0}^\infty \psi_j \equiv 1$ and for every multiindex $\alpha\in \N^d$ there exists a constant $c_\alpha$ such that
\begin{equation*}
\norm{D^\alpha \psi_j}_\infty \leq \frac{c_\alpha}{2^{j (\alpha_1+\cdots +\alpha_d)} } \mbox{\,\,\, for every $j\geq 0$}.
\end{equation*}

We will use the classical notation $\widehat f$ for the Fourier transform of a given Schwartz function,
$$\widehat f (\xi)=\int_{\R^d} e^{-2\pi i x\cdot \xi} f(x)\, dx,$$ 
and $\widecheck f$ will denote its inverse.
It is well known that the Fourier transform can be extended to the whole space of tempered distributions by duality and it induces an isometry in $L^2$ (see for example \cite[Chapter 2]{Grafakos}).

 \begin{definition}
Let $s \in \R$,  $0< p< \infty$, $0< q\leq\infty$. For any tempered distribution $f\in S'(\R^d)$ we define its non-homogeneous Triebel-Lizorkin quasi-norm
$$\norm{f}_{F^s_{p,q}}=\norm{\norm{\left\{2^{sj}\left(\psi_j \widehat{f}\right)\widecheck{\,}\right\}}_{l^q}}_{L^p},$$
and we call $F^s_{p,q}\subset S'$ to the set of tempered distributions such that this quasi-norm is finite. 
\end{definition}

These quasi-norms (norms when $p,q \geq 1$) are equivalent for different choices of $\{\psi_j\}_{j=0}$ (see \cite[Section 2.3]{TriebelTheory}). Changing the order of integration and summation above we get the non-homogeneous Besov quasi-norm
$$\norm{f}_{B^s_{p,q}}=\norm{\left\{2^{sj}\norm{\left(\psi_j \widehat{f}\right)\widecheck{\,} }_{L^p}\right\}}_{l^q},$$
which makes sense for $0< p\leq \infty$. 

For $q=2$ and $1<p<\infty$ the Triebel-Lizorkin spaces coincide with the so-called Bessel-potential spaces.
  In addition, if $s\in \N$ they coincide with the usual Sobolev spaces of functions in $L^p$ with weak derivatives up to order $s$ in $L^p$,  and they coincide with $L^p$ for $s=0$ (\cite[Section 2.5.6]{TriebelTheory}). In the present text, we use the convention 
\begin{equation*}
 W^{s,p}:= F^s_{p,2} \mbox{ \quad\quad for } s\geq 0 \mbox{ and } 1<p<\infty.
\end{equation*}
and for $s=0$ we write
\begin{equation*}
h^p:=F^0_{p,2} \mbox{ \quad\quad for } 0<p<\infty,
\end{equation*}
that is, the non-homogeneous hardy space (which coincides with $L^p$ for $1<p<\infty$). With this convention, complex interpolation between Sobolev spaces is a Sobolev space (see \cite[Section 2.4.2, Theorem 1]{TriebelInterpolation}). For more information on the relation between Triebel-Lizorkin and Besov spaces with other classical spaces we refer the reader to \cite[Section 2.2.2]{TriebelTheory} and \cite[Theorem 2.2.2]{RunstSickel}. 

There is a whole structure of embeddings for Triebel-Lizorkin spaces (see  \cite[Proposition 2.3.2 and Theorem 2.7.1]{TriebelTheory}, \cite[Chapter 1]{RunstSickel} for end-point cases). For instance, if $-\infty<s<\infty$, $0< p < \infty$, $0< q_0, q_1 \leq \infty$ and $\varepsilon>0$, we have that 
\begin{equation}\label{eqTrivialEmbedding}
F^s_{p,q_{0}}\subset F^s_{p,q_{0} + \varepsilon} \quad \quad \mbox{ and } \quad\quad F^{s+\varepsilon}_{p,q_0}\subset F^{s}_{p,q_1}
\end{equation}
and if $-\infty<s_1<s_0<\infty$, $0< p_0 < p_1 < \infty$ satisfy that $s_0-\frac{d}{p_0}=s_1-\frac{d}{p_1}$, for $0< q_0, q_1 \leq \infty$ we have that
\begin{equation}\label{eqSobolevEmbedding}
F^{s_0}_{p_0,q_0}\subset F^{s_1}_{p_1,q_1} .
\end{equation}
(see Figure \ref{figEmbeddingsGeneral}). Besov spaces present a similar structure.

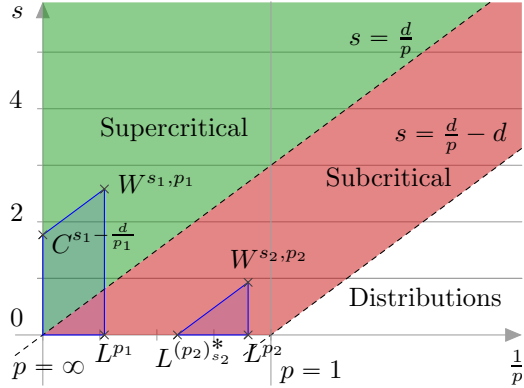
\begin{figure}[ht]
\centering
\begin{tikzpicture}[line cap=round,line join=round,>=triangle 45,x=3.0cm,y=0.75cm]
\draw[->,color=grisclar] (-0.12,0.) -- (2.1,0.);
\foreach \x in {0,0.25,0.5,0.75,1.}
\draw[shift={(\x,0)},color=grisclar] (0pt,2pt) -- (0pt,-2pt);
 { \draw[->,color=grisclar] (0.,-0.8884766569124901) -- (0.,5.891712114631426);};
 {\foreach \y in {,1.,2.,3.,4.,5.}\draw[shift={(0,\y)},color=grisclar] (2.1,0.) -- (-2pt,0.);};
{\clip(-0.2209869282157063,-0.8884766569124901) rectangle (2.1,6);};
{\fill[color=vermell,fill=vermell,fill opacity=0.5] (0.,0.) -- (1.,0.) -- (2.1,3.3) -- (2.1,5.88) -- (1.96,5.88) -- cycle;};
{\fill[color=verd,fill=verd,fill opacity=0.5] (0.,0.) -- (1.96,5.88) -- (0.,5.88) -- cycle;};
{\fill[color=blauclar,fill=blauclar,fill opacity=0.1]  (0.27,2.58)-- (0.,1.77) -- (0.,0.) -- (0.27,0.) -- cycle;};
{\fill[color=blauclar,fill=blauclar,fill opacity=0.1]  (0.9,0.93) -- (0.59,0.)-- (0.9,0.)-- cycle;};
{\draw (2,-0.1) node[anchor=north west] {$\frac{1}{p}$};};
{\draw [dash pattern=on 2pt off 2pt,domain=-0.1209869282157063:1.96] plot(\x,{(-0.--3.*\x)/1.});};
{\draw [dash pattern=on 2pt off 2pt,domain=0.88:2.1] plot(\x,{(-3.--3.*\x)/1.});};
{\draw [color=grisclar] (1.,-0.8884766569124901) -- (1.,5.891712114631426);};
{\draw (-0.18062866443410706,5.972428647625997) node[anchor=north west] {$s$};};
{\draw (1,-0.3) node[anchor=north west] {$p=1$};};
{\draw (-0.1725570116777872,-0.3) node[anchor=north west] {$p=\infty$};};
{\draw (1.3,5.7) node[anchor=north west] {$s=\frac{d}{p}$};};
{\draw (1.5,4.) node[anchor=north west] {$s=\frac{d}{p}-d$};};
{\draw (0.21,4) node[anchor=north west] {$\mbox{Supercritical}$};};
{\draw (1.2,3.12) node[anchor=north west] {$\mbox{Subcritical}$};};
{\draw (1.3,1) node[anchor=north west] {$\mbox{Distributions}$};};
{\draw (-0.18466449081226698,0.6249583367357056) node[anchor=north west] {$0$};};
{\draw (-0.18466449081226698,2.4814385956108254) node[anchor=north west] {$2$};};
{\draw (-0.1887003171904269,4.479172787226443) node[anchor=north west] {$4$};};
{\draw [color=blauclar] (0.27,2.58)-- (0.,1.77);};
{\draw [color=blauclar] (0.,1.77)-- (0.,0.);};
{\draw [color=blauclar] (0.,0.)-- (0.27,0.);};
{\draw [color=blauclar] (0.27,0.)-- (0.27,2.58);};
{\draw (0.286631535978446,3) node[anchor=north west] {$W^{s_1,p_1}$};};
{\draw (-0.003947963249068506,2.1383933303839013) node[anchor=north west] {$C^{s_1-\frac{d}{p_1}}$};};
{\draw (0.1857358765244479,0) node[anchor=north west] {$L^{p_1}$};};
{\draw [color=blauclar] (0.9,0.93)-- (0.59,0.);};
{\draw [color=blauclar] (0.59,0.)-- (0.9,0.);};
{\draw [color=blauclar] (0.9,0.)-- (0.9,0.93);};
{\draw (0.7790023541139567,1.6742732656651211) node[anchor=north west] {$W^{s_2,p_2}$};};
{\draw (0.45,0.1) node[anchor=north west] {$L^{(p_2)^*_{s_2}}$};};
{\draw (0.8395397497863556,0) node[anchor=north west] {$L^{p_2}$};};
\begin{scriptsize}
{\draw [color=grisfosc] (0.27,2.58)-- ++(-1.5pt,-1.5pt) -- ++(3.0pt,3.0pt) ++(-3.0pt,0) -- ++(3.0pt,-3.0pt);};
{\draw [color=grisfosc] (0.9,0.93)-- ++(-1.5pt,-1.5pt) -- ++(3.0pt,3.0pt) ++(-3.0pt,0) -- ++(3.0pt,-3.0pt);};
{\draw [color=grisfosc] (0.27,0.)-- ++(-1.5pt,-1.5pt) -- ++(3.0pt,3.0pt) ++(-3.0pt,0) -- ++(3.0pt,-3.0pt);};
{\draw [color=grisfosc] (0.9,0.)-- ++(-1.5pt,-1.5pt) -- ++(3.0pt,3.0pt) ++(-3.0pt,0) -- ++(3.0pt,-3.0pt);};
{\draw [color=grisfosc] (0.59,0.)-- ++(-1.5pt,-1.5pt) -- ++(3.0pt,3.0pt) ++(-3.0pt,0) -- ++(3.0pt,-3.0pt);};
{\draw [color=grisfosc] (0.,1.77)-- ++(-1.5pt,-1.5pt) -- ++(3.0pt,3.0pt) ++(-3.0pt,0) -- ++(3.0pt,-3.0pt);};
\end{scriptsize}
\end{tikzpicture}
\caption{General embeddings for Sobolev  spaces (and Triebel-Lizorkin spaces with $q$ fixed) in dimension $d=3$ (see  \rf{eqTrivialEmbedding}, \rf{eqSobolevEmbedding} and subsequent embeddings).}\label{figEmbeddingsGeneral}
\end{figure}

Regarding classical spaces, whenever $0< p < \infty$ and $0< q< \infty$, the following holds true:
\begin{itemize}
\item If $s>\frac dp$ (supercritical case) then $F^s_{p,q}\subset \mathcal{C}^{s-\frac{d}{p}}\cap h^p$.
\item If $s=\frac dp$ (critical case), then $F^s_{p,q}\subset VMO\cap h^p$.
\item If $\frac dp -d <s<\frac dp$ (subcritical case), then $F^s_{p,q}\subset L^{(p)^*_s}\cap h^p$ where $\frac{1}{(p)^*_s}=\frac{1}{p}-\frac{s}{d}<1$. 
\end{itemize}

%NO ESBORRAR: Escala completa de Triebel,  Theory of Function Spaces I
%\begin{itemize} 
%\item if $s=(1-\Theta)s_0 + \Theta s_1$, then
%$$(B^s_{p,q_0},B^s_{p,q_1})_{\Theta,q} = B^s_{p,q},$$
%\item if $s=(1-\Theta)s_0 + \Theta s_1$ and $p<\infty$, then
%$$(F^s_{p,q_0},F^s_{p,q_1})_{\Theta,q} = B^s_{p,q},$$
%\item if $s=(1-\Theta)s_0 + \Theta s_1$  and $\frac{1}{p}=\frac{1-\Theta}{p_0}+\frac{\Theta}{p_1}$ with $p_0,p_1<\infty$ then
%$$(B^{s_0}_{p_0,p_0},B^{s_1}_{p_1,p_1})_{\Theta,p}=B^s_{p,p},$$
%\item if $s=(1-\Theta)s_0 + \Theta s_1$, $\frac{1}{p}=\frac{1-\Theta}{p_0}+\frac{\Theta}{p_1}$ and $\frac{1}{q}=\frac{1-\Theta}{q_0}+\frac{\Theta}{q_1}$ then
%$$(B^{s_0}_{p_0,q_0},B^{s_1}_{p_1,q_1})_{\Theta} = B^s_{p,q},$$
%\item if $s=(1-\Theta)s_0 + \Theta s_1$, $\frac{1}{p}=\frac{1-\Theta}{p_0}+\frac{\Theta}{p_1}$ with $p_0,p_1<\infty$ and $\frac{1}{q}=\frac{1-\Theta}{q_0}+\frac{\Theta}{q_1}$ then
%$$(F^{s_0}_{p_0,q_0},F^{s_1}_{p_1,q_1})_{\Theta} = F^s_{p,q}.$$
% \end{itemize}

\subsection{Compactly supported functions}\label{secCompactlySupported}
Let us assume that $\mu \in W^{s,p}$  is compactly supported with $p>1$. Since $\mu \in L^{(p)^*_s}$ (note that one can do the same if $\mu \in F^s_{p,q}$ when $s-\frac{d}{p}>-d$ as we noted above), which coincides with the Hardy space $h_p$, also $\mu \in h_r$ for $0<r<p$ (see \cite[Section 2.2.2]{TriebelTheory}). By interpolation, we have that $\mu \in W^{\sigma,r}$ as well for $\sigma<s$ and $1<r\leq p$. Combined with the embeddings described in \rf{eqTrivialEmbedding} and \rf{eqSobolevEmbedding}, this gives us an almost complete picture of the spaces where $\mu$ does belong. It remains to see what happens in the endpoint $\sigma = s$ and in the remaining spaces of the scale $F^s_{p,q}$. If $s$ is a natural number, then the derivatives of order $s$ are compactly supported as well and it follows that $\mu\in W^{s,r}$ for every $0<r\leq p$, but in case $s$ is not a natural number, some extra work needs to be done. Since  the author does not know any mention in the literature about this case, it is studied in this section. 

We will argue using expressions in terms of differences.  Let us write
$\Delta_h^1 f(x):=f(x+h)-f(x)$ and, if $M\in\N$ with $M>1$ we define the $M$-th iterated difference as $\Delta_h^Mf(x):=\Delta_h^1 (\Delta_h^{M-1} f)(x) = \sum_{j=0}^M{M\choose j} (-1)^{M-j} f(x+jh)$. 

\begin{lemma}\label{lemCompact}
Let $d\in\N$, $0<s<\infty$, $\frac{d}{d+s}<p_1 \leq p_0<\infty$ and $\frac{d}{d+s}<q\leq\infty$.  For every compactly supported function $f\in F^s_{p_0,q}$, we have that $f\in F^s_{p_1,q}$.
\end{lemma}

\begin{proof}
Let $f\in F^s_{p_0,q}$ be given with compact support in $\DDD_R$. Choose a natural number $M>s$.  Clearly 
$$d_{t}^M f(x):=t^{-d}\int_{|h|\leq t}|\Delta_h^M f(x)|\, dh ,$$
is compactly supported in the disk $\DDD_{R+M}$ for $t\leq 1$,  and so is 
$$w_{q}^M f(x):=\left(\int_0^1  \frac{d_{t}^M f(x)^q}{t^{sq+1}} \, dt \right)^{\frac{1}{q}},$$
(with the usual modifications for $q=\infty$). 

By hypothesis we have $\frac{d}{\min\{p_0,p_1,q\}}-d<s$. By \cite[Theorem 1.116]{TriebelTheoryIII},  writing $\overline{p_{j}}:=\max\{1,p_{j}\}$ for every locally integrable function $g$ we have that
\begin{equation}\label{eqTriebelDifference}
\norm{g}_{F^s_{p_{j},q}(\R^d)}\approx \norm{g}_{L^{p_{j}}}+ \norm{g}_{L^{\overline{p_{j}}}} + \norm{w_{q}^M g}_{L^{p_j}} .
\end{equation}

Now, the norm in \rf{eqTriebelDifference} and the H\"older inequality grant that
 \begin{align*}
 \norm{f}_{F^s_{p_1,q}}
 	& \approx \norm{f}_{L^{\overline{p_1}}(\DDD_R)}+ \norm{f}_{L^{p_1}(\DDD_R)} + \norm{w_{q}^M f}_{L^{p_1}(\DDD_{R+M})} \\
 	& \lesssim_R  \norm{f}_{L^{\overline{p_0}}} + \norm{f}_{L^{p_0}} + \norm{w_{q}^M f}_{L^{p_0}} 	
		\approx   \norm{f}_{F^s_{p_0,q}}.
\end{align*}
\end{proof}
\begin{remark}
Note that the one can obtain analogous results for Besov spaces using also norms in terms of differences, with the more general assumption $0<q\leq \infty$, $\frac{d}{d+s}<p_1 \leq p_0 \leq \infty$. The norm in \cite[(2.5.12.4)]{TriebelTheory} by replacing $\int_{\R^d}$ by $\int_{|h|\leq 1}$, for instance, will do the job.
\end{remark}

\begin{figure}[ht]
\centering
\subfigure[Embeddings for compactly supported functions.]
{
\begin{tikzpicture}[line cap=round,line join=round,>=triangle 45,x=3.0cm,y=1.12cm]
\draw[->,color=grisclar] (-0.12,0.) -- (2.1,0.);
\foreach \x in {0,0.25,0.5,0.75,1.}
\draw[shift={(\x,0)},color=grisclar] (0pt,2pt) -- (0pt,-2pt);
 { \draw[->,color=grisclar] (0.,-0.88) -- (0.,4.1);};
 {\foreach \y in {,1.,2.,3.,4.}\draw[shift={(0,\y)},color=grisclar] (2.1,0.) -- (-2pt,0.);};
{\clip(-0.22,-0.88) rectangle (2.1,4.1);};
{\fill[color=vermell,fill=vermell,fill opacity=0.5] (0.,0.) -- (1.,0.) -- (2.1,2.2) -- (2.1,4.1) -- (2.05,4.1) -- cycle;};
{\fill[color=verd,fill=verd,fill opacity=0.5] (0.,0.) -- (2.05,4.1) -- (0.,4.1) -- cycle;};
{\fill[color=blauclar,fill=blauclar,fill opacity=0.1] (0.3,1.9)-- (0.,1.2) -- (0.,0.) -- (2.1,0.) -- (2.1,1.9) -- cycle;};
{\fill[color=blauclar,fill=blauclar,fill opacity=0.1]  (0.9,0.62) -- (0.59,0.)-- (2.1,0.) -- (2.1,0.62) -- cycle;};
{\draw (1.9,-0.1) node[anchor=north west] {$\frac{1}{p}$};};
{\draw [dash pattern=on 2pt off 2pt,domain=-0.12:2.05] plot(\x,{(-0.--2.*\x)/1.});};
{\draw [dash pattern=on 2pt off 2pt,domain=0.88:2.1] plot(\x,{(-2.--2.*\x)/1.});};
{\draw [color=grisclar] (1.,-0.88) -- (1.,4.1);};
{\draw (-0.18,4.1) node[anchor=north west] {$s$};};
{\draw (1,-0.3) node[anchor=north west] {$p=1$};};
{\draw (-0.17,-0.3) node[anchor=north west] {$p=\infty$};};
{\draw (1.5,4.) node[anchor=north west] {$s=\frac{d}{p}$};};
{\draw (1.5,2.5) node[anchor=north west] {$s=\frac{d}{p}-d$};};
{\draw (-0.18466449081226698,0.6249583367357056) node[anchor=north west] {$0$};};
{\draw (-0.18466449081226698,2.4814385956108254) node[anchor=north west] {$2$};};
{\draw [color=blauclar] (0.3,1.9)-- (0.,1.2);};
{\draw [color=blauclar] (0.,1.2)-- (0.,0.);};
{\draw [color=blauclar] (0.,0.)-- (2.1,0.);};
{\draw [color=blauclar] (1.95,1.9)-- (0.3,1.9);};
{\draw (0.28,2.5) node[anchor=north west] {$W^{s_1,p_1}$};};
{\draw (0.,1.5) node[anchor=north west] {$C^{s_1-\frac{d}{p_1}}$};};
{\draw [color=blauclar] (0.9,0.62)-- (0.59,0.);};
{\draw [color=blauclar] (0.59,0.)-- (2.1,0.);};
{\draw [color=blauclar] (1.31,0.62)-- (0.9,0.62);};
{\draw (0.77,1.) node[anchor=north west] {$W^{s_2,p_2}$};};
{\draw (0.45,0.1) node[anchor=north west] {$L^{(p_2)^*_{s_2}}$};};
\begin{scriptsize}
{\draw [color=grisfosc] (0.3,1.9)-- ++(-1.5pt,-1.5pt) -- ++(3.0pt,3.0pt) ++(-3.0pt,0) -- ++(3.0pt,-3.0pt);};
{\draw [color=grisfosc] (0.9,0.62)-- ++(-1.5pt,-1.5pt) -- ++(3.0pt,3.0pt) ++(-3.0pt,0) -- ++(3.0pt,-3.0pt);};
{\draw [color=grisfosc] (0.59,0.)-- ++(-1.5pt,-1.5pt) -- ++(3.0pt,3.0pt) ++(-3.0pt,0) -- ++(3.0pt,-3.0pt);};
{\draw [color=grisfosc] (0.,1.2)-- ++(-1.5pt,-1.5pt) -- ++(3.0pt,3.0pt) ++(-3.0pt,0) -- ++(3.0pt,-3.0pt);};
\end{scriptsize}
\end{tikzpicture}
}\,\,
\subfigure[Embeddings for bounded and compactly supported functions.]
{
\begin{tikzpicture}[line cap=round,line join=round,>=triangle 45,x=3.0cm,y=1.12cm]
\draw[->,color=grisclar] (-0.12,0.) -- (2.1,0.);
\foreach \x in {0,0.25,0.5,0.75,1.}
\draw[shift={(\x,0)},color=grisclar] (0pt,2pt) -- (0pt,-2pt);
 { \draw[->,color=grisclar] (0.,-0.88) -- (0.,4.1);};
 {\foreach \y in {,1.,2.,3.,4.}\draw[shift={(0,\y)},color=grisclar] (2.1,0.) -- (-2pt,0.);};
{\clip(-0.22,-0.88) rectangle (2.1,4.1);};
{\fill[color=vermell,fill=vermell,fill opacity=0.5] (0.,0.) -- (1.,0.) -- (2.1,2.2) -- (2.1,4.1) -- (2.05,4.1) -- cycle;};
{\fill[color=verd,fill=verd,fill opacity=0.5] (0.,0.) -- (2.05,4.1) -- (0.,4.1) -- cycle;};
{\fill[color=blauclar,fill=blauclar,fill opacity=0.1] (0.3,1.9)-- (0.,1.2) -- (0.,0.) -- (2.1,0.) -- (2.1,1.9) -- cycle;};
{\fill[color=blauclar,fill=blauclar,fill opacity=0.1]  (0.9,0.62) -- (0.,0.)-- (2.1,0.) -- (2.1,0.62) -- cycle;};
{\draw (1.9,-0.1) node[anchor=north west] {$\frac{1}{p}$};};
{\draw [dash pattern=on 2pt off 2pt,domain=-0.12:2.05] plot(\x,{(-0.--2.*\x)/1.});};
{\draw [dash pattern=on 2pt off 2pt,domain=0.88:2.1] plot(\x,{(-2.--2.*\x)/1.});};
{\draw [color=grisclar] (1.,-0.88) -- (1.,4.1);};
{\draw (-0.18,4.1) node[anchor=north west] {$s$};};
{\draw (1,-0.3) node[anchor=north west] {$p=1$};};
{\draw (-0.17,-0.3) node[anchor=north west] {$p=\infty$};};
{\draw (1.5,4.) node[anchor=north west] {$s=\frac{d}{p}$};};
{\draw (1.5,2.5) node[anchor=north west] {$s=\frac{d}{p}-d$};};
{\draw (-0.18466449081226698,0.6249583367357056) node[anchor=north west] {$0$};};
{\draw (-0.18466449081226698,2.4814385956108254) node[anchor=north west] {$2$};};
{\draw [color=blauclar] (0.3,1.9)-- (0.,1.2);};
{\draw [color=blauclar] (0.,1.2)-- (0.,0.);};
{\draw [color=blauclar] (0.,0.)-- (2.1,0.);};
{\draw [color=blauclar] (1.95,1.9)-- (0.3,1.9);};
{\draw (0.28,2.5) node[anchor=north west] {$W^{s_1,p_1}$};};
{\draw (0.,1.5) node[anchor=north west] {$C^{s_1-\frac{d}{p_1}}$};};
{\draw [color=blauclar] (0.9,0.62)-- (0,0.);};
{\draw [color=blauclar] (0,0.)-- (2.1,0.);};
{\draw [color=blauclar] (1.31,0.62)-- (0.9,0.62);};
{\draw (0.77,1.) node[anchor=north west] {$W^{s_2,p_2}$};};
{\draw (0.,0.) node[anchor=north west] {$L^{\infty}$};};
\begin{scriptsize}
{\draw [color=grisfosc] (0.3,1.9)-- ++(-1.5pt,-1.5pt) -- ++(3.0pt,3.0pt) ++(-3.0pt,0) -- ++(3.0pt,-3.0pt);};
{\draw [color=grisfosc] (0.9,0.62)-- ++(-1.5pt,-1.5pt) -- ++(3.0pt,3.0pt) ++(-3.0pt,0) -- ++(3.0pt,-3.0pt);};
{\draw [color=grisfosc] (0.,0.)-- ++(-1.5pt,-1.5pt) -- ++(3.0pt,3.0pt) ++(-3.0pt,0) -- ++(3.0pt,-3.0pt);};
{\draw [color=grisfosc] (0.,1.2)-- ++(-1.5pt,-1.5pt) -- ++(3.0pt,3.0pt) ++(-3.0pt,0) -- ++(3.0pt,-3.0pt);};
\end{scriptsize}
\end{tikzpicture}
}
\caption{Embeddings for compactly supported or bounded functions.}\label{figSpecialEmbeddings}
\end{figure}
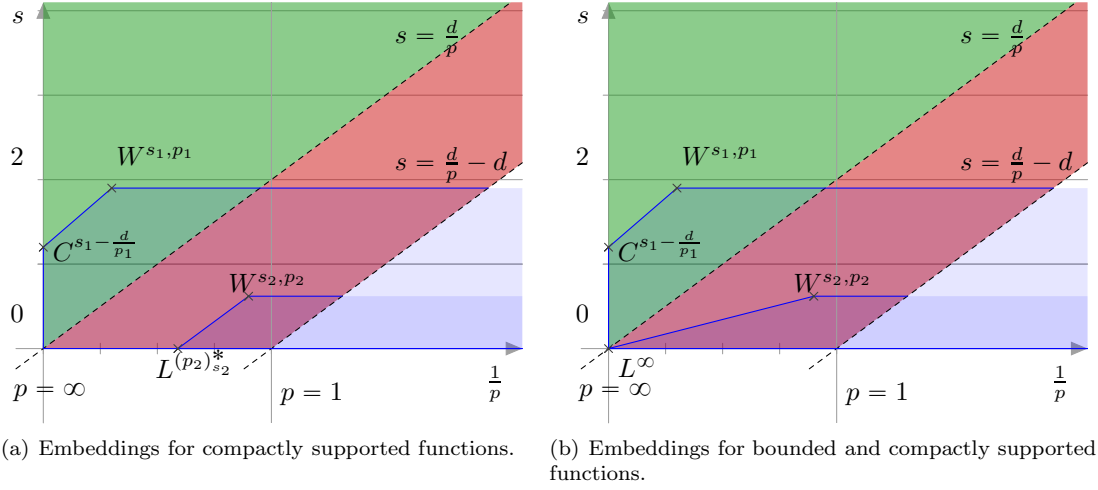

Finally let us compile all the information regarding functions such as the Beltrami coefficients involved in Theorem \ref{theoRegularityCoefficients} (see Figure \ref{figSpecialEmbeddings} (b)). In this case, we can interpolate norms with $L^{\infty}$ by \cite[Theorem 2.2.5]{RunstSickel}: Let  $0<s_2 < s<\frac dp$, $\frac{d}{d+s} < p_1 \leq  p <\infty$ and $0<p_2<\infty$, $0<q\leq q_1 \leq\infty$ with $\frac{d}{d+s}<q_1$ and $0<q_2\leq\infty$. For every compactly supported function $\mu\in F^s_{p,q}\cap L^\infty$ we have that 
\begin{equation}\label{eqInterpolateLInfty}
\mu \in F^{s}_{p_1,q_1} \cap F^{s_2}_{p_2,q_2}  \mbox{\quad\quad as long as \quad\quad} s_2 p_2 \leq sp .
\end{equation}

\subsection{Discussion on the former regularity results}\label{secPreviousResults}
\begin{figure}[ht]
\centering
\subfigure[{General case from \cite{Astala}},  for $q<p_\kappa$, $\norm{\mu}_{L^\infty} \leq \kappa \implies  f \in W^{1,q}_{loc}$. ]
{
\begin{tikzpicture}[line cap=round,line join=round,>=triangle 45,x=3.0cm,y=1.12cm]
\draw[->,color=grisclar] (-0.12,0.) -- (1.1,0.);
\foreach \x in {0,0.25,0.5,0.75,1.}
\draw[shift={(\x,0)},color=grisclar] (0pt,2pt) -- (0pt,-2pt);
 { \draw[->,color=grisclar] (0.,-0.88) -- (0.,3.1);};
 {\foreach \y in {,1.,2.,3.}\draw[shift={(0,\y)},color=grisclar] (1.1,0.) -- (-2pt,0.);};
{\clip(-0.22,-0.88) rectangle (1.1,3.1);};
%{\fill[color=vermell,fill=vermell,fill opacity=0.5] (0.,0.) -- (1.,0.) -- (1.1,0.2) -- (1.1,2.2) -- cycle;};
%{\fill[color=verd,fill=verd,fill opacity=0.5] (0.,0.) -- (1.1,2.2) -- (1.1,4.1) -- (0.,4.1) -- cycle;};
{\draw (1.0,-0.1) node[anchor=north west] {$\frac{1}{p}$};};
{\draw [color=grisclar, dash pattern=on 2pt off 2pt,domain=-0.12:1.1] plot(\x,{(-0.--2.*\x)/1.});};
{\draw [color=grisclar] (1.,-0.88) -- (1.,3.1);};
{\draw [color=negre, dash pattern=on 1pt off 3pt] (0.3,0) -- (0.3,1);};
{\draw (-0.18,3.1) node[anchor=north west] {$s$};};
{\node [circle, draw=negre, thick,inner sep=0pt, minimum size=1pt] (1) at (0.7,0) {$\,$};};
{\draw (0.43,-0.07) node[anchor=north east] {$\frac{1}{p_\kappa}$};};
{\node [circle, draw=negre, thick,inner sep=0pt, minimum size=1pt] (2) at (0.3,0) {$\,$};};
{\draw (0.83,-0.07) node[anchor=north east] {$\frac{1}{p_\kappa'}$};};
{\draw [color=grisfosc] (0.7,3) node[anchor=north west] {$W^{s,p}_{loc}$};};
{\draw (1.12,-0.1) node[anchor=north west] {$\frac{1}{p}$};};
{\draw[color=red] (0.3,1.) node[anchor=south west] {$f$};};
{\fill[color=red,fill=vermell,fill opacity=0.5] (0.,0.) -- (1.,0.) -- (1.,1.) -- (0.3,1.) -- (0,0.4) -- cycle;};
{\draw[color=red] (0,0.4) -- (0.,0.) -- (1.,0.) -- (1.,1.) -- (0.3,1.) ;};
{\draw[color=red, dash pattern=on 2pt off 2pt] (0,0.4) -- (0.3,1.) ;};
{\draw [color=blau] (0,0) node[anchor=north west] {$\mu$};};
{\draw [color=blau, thick] (0,0) -- (1,0);};
{\draw [color= blauverd] (0.25,0.5) node[anchor=north west] {$\bar\partial f$};};
{\draw [color= blauverd, thick] (0.3,-0.025) -- (1,-0.025);};
\begin{scriptsize}
{\node [circle, draw=red,inner sep=0pt, minimum size=2pt] (1) at (0.3,1) {$\,$};};
{\draw [color=blau, thick] (0,0.)-- ++(-1.5pt,-1.5pt) -- ++(3.0pt,3.0pt) ++(-3.0pt,0) -- ++(3.0pt,-3.0pt);};
{\node [circle, fill=white, inner sep=0pt, minimum size=2pt] (1) at (0.3,0) {$\,$};};
{\node [circle, draw=blauverd,fill=vermell,fill opacity=0.5, inner sep=0pt, minimum size=2pt] (1) at (0.3,0) {$\,$};};
\end{scriptsize}
\end{tikzpicture}
}\,\,
\subfigure[From \cite{Iwaniec},  if  $1\leq q<\infty$, $\mu \in VMO \implies  f \in W^{1,q}_{loc}$.]
{
\begin{tikzpicture}[line cap=round,line join=round,>=triangle 45,x=3.0cm,y=1.12cm]
\draw[->,color=grisclar] (-0.12,0.) -- (1.1,0.);
\foreach \x in {0,0.25,0.5,0.75,1.}
\draw[shift={(\x,0)},color=grisclar] (0pt,2pt) -- (0pt,-2pt);
 { \draw[->,color=grisclar] (0.,-0.88) -- (0.,3.1);};
 {\foreach \y in {,1.,2.,3.}\draw[shift={(0,\y)},color=grisclar] (1.1,0.) -- (-2pt,0.);};
{\clip(-0.22,-0.88) rectangle (1.1,3.1);};
%{\fill[color=vermell,fill=vermell,fill opacity=0.5] (0.,0.) -- (1.,0.) -- (1.1,0.2) -- (1.1,2.2) -- cycle;};
%{\fill[color=verd,fill=verd,fill opacity=0.5] (0.,0.) -- (1.1,2.2) -- (1.1,4.1) -- (0.,4.1) -- cycle;};
{\draw (1.0,-0.1) node[anchor=north west] {$\frac{1}{p}$};};
{\draw [color=grisclar, dash pattern=on 2pt off 2pt,domain=-0.12:1.1] plot(\x,{(-0.--2.*\x)/1.});};
{\draw [color=grisclar] (1.,-0.88) -- (1.,3.1);};
{\draw (-0.18,3.1) node[anchor=north west] {$s$};};
{\node [circle, draw=negre, thick,inner sep=0pt, minimum size=1pt] (1) at (0.7,0) {$\,$};};
{\draw (0.43,-0.07) node[anchor=north east] {$\frac{1}{p_\kappa}$};};
{\node [circle, draw=negre, thick,inner sep=0pt, minimum size=1pt] (2) at (0.3,0) {$\,$};};
{\draw (0.83,-0.07) node[anchor=north east] {$\frac{1}{p_\kappa'}$};};
{\draw [color=grisfosc] (0.7,3) node[anchor=north west] {$W^{s,p}_{loc}$};};
{\draw (1.12,-0.1) node[anchor=north west] {$\frac{1}{p}$};};
{\draw[color=red] (0,1.) node[anchor=south west] {$f$};};
{\fill[color=red,fill=vermell,fill opacity=0.5] (0.,0.) -- (1.,0.) -- (1.,1.) -- (0,1.) --cycle;};
{\draw[color=red] (0,0.4) -- (0.,0.) -- (1.,0.) -- (1.,1.) -- (0.,1.) -- cycle;};
{\draw [color=blau] (0,0) node[anchor=north west] {$\mu$};};
{\node [circle, draw=red,inner sep=0pt, minimum size=2pt] (1) at (0,1) {$\,$};};
{\draw [color=blau, thick] (0,0) -- (1,0);};
{\draw [color= blauverd] (0.,0.) node[anchor=south west] {$\bar\partial f$};};
{\draw [color= blauverd, thick] (0,-0.025) -- (1,-0.025);};
\begin{scriptsize}
{\draw [color=blau, thick] (0,0.)-- ++(-1.5pt,-1.5pt) -- ++(3.0pt,3.0pt) ++(-3.0pt,0) -- ++(3.0pt,-3.0pt);};
{\node [circle, draw=blauverd,inner sep=0pt, minimum size=2pt] (1) at (0,0) {$\,$};};
{\node [circle, draw=red,inner sep=0pt, minimum size=2pt] (1) at (0,1) {$\,$};};
\end{scriptsize}
\end{tikzpicture}
}\,\,
\subfigure[{From \cite[Chapter 15]{AstalaIwaniecMartin})}, $\mu \in C^s \implies f \in C^{s+1}_{loc}$.]
{
\begin{tikzpicture}[line cap=round,line join=round,>=triangle 45,x=3.0cm,y=1.12cm]
\draw[->,color=grisclar] (-0.12,0.) -- (1.1,0.);
\foreach \x in {0,0.25,0.5,0.75,1.}
\draw[shift={(\x,0)},color=grisclar] (0pt,2pt) -- (0pt,-2pt);
 { \draw[->,color=grisclar] (0.,-0.88) -- (0.,3.1);};
 {\foreach \y in {,1.,2.,3.}\draw[shift={(0,\y)},color=grisclar] (1.1,0.) -- (-2pt,0.);};
{\clip(-0.22,-0.88) rectangle (1.1,3.1);};
{\draw (1.0,-0.1) node[anchor=north west] {$\frac{1}{p}$};};
{\draw [color=grisclar, dash pattern=on 2pt off 2pt,domain=-0.12:1.1] plot(\x,{(-0.--2.*\x)/1.});};
{\draw [color=grisclar] (1.,-0.88) -- (1.,3.1);};
{\draw (-0.18,3.1) node[anchor=north west] {$s$};};
{\node [circle, draw=negre, thick,inner sep=0pt, minimum size=1pt] (1) at (0.7,0) {$\,$};};
{\draw (0.43,-0.07) node[anchor=north east] {$\frac{1}{p_\kappa}$};};
{\node [circle, draw=negre, thick,inner sep=0pt, minimum size=1pt] (2) at (0.3,0) {$\,$};};
{\draw (0.83,-0.07) node[anchor=north east] {$\frac{1}{p_\kappa'}$};};
{\draw [color=grisfosc] (0.5,3) node[anchor=north west] {$B^{s}_{\infty,\infty,loc}$};};
{\draw (1.12,-0.1) node[anchor=north west] {$\frac{1}{p}$};};
{\draw[color=red] (0,2.4) node[anchor=south west] {$f$};};
{\fill[color=red,fill=vermell,fill opacity=0.5] (0.,0.) -- (1.,0.) -- (1.,2.4) -- (0,2.4) -- cycle;};
{\draw[color=red] (0.,0.) -- (1.,0.) -- (1.,2.4) -- (0,2.4) -- cycle;};
{\draw [color=blau] (0,1.4) node[anchor=north east] {$\mu$};};
{\draw[color= blau] (0.,0.) -- (1.,0.) -- (1.,1.4) -- (0,1.4) -- cycle;};
{\draw [color= blauverd] (0,1.4) node[anchor=south east] {$\bar\partial f$};};
{\draw [color= blauverd, thick] (0,1.375) -- (1,1.375);};
{\draw [color= blauverd,  dash pattern=on 2pt off 2pt, thick] (0.0075,0) -- (0.0075,1.375);};
\begin{scriptsize}
{\draw [color=blau, thick] (0,1.4)-- ++(-1.5pt,-1.5pt) -- ++(3.0pt,3.0pt) ++(-3.0pt,0) -- ++(3.0pt,-3.0pt);};
{\draw [color=blauverd]  (0,1.4)-- ++(-1pt,-1pt) -- ++(2.0pt,2.0pt) ++(-2.0pt,0) -- ++(2.0pt,-2.0pt);};
{\draw [color=red] (0,2.4)-- ++(-1.5pt,-1.5pt) -- ++(3.0pt,3.0pt) ++(-3.0pt,0) -- ++(3.0pt,-3.0pt);};
\end{scriptsize}
\end{tikzpicture}
}

\subfigure[From \cite{ClopFaracoMateuOrobitgZhong}, if $p_\kappa'<p<2$, $\mu\in W^{1,p} \implies \bar\partial f \in W^{1,q}$, when $q<<p$.]
{
\begin{tikzpicture}[line cap=round,line join=round,>=triangle 45,x=3.0cm,y=1.12cm]
\draw[->,color=grisclar] (-0.12,0.) -- (1.1,0.);
\foreach \x in {0,0.25,0.5,0.75,1.}
\draw[shift={(\x,0)},color=grisclar] (0pt,2pt) -- (0pt,-2pt);
 { \draw[->,color=grisclar] (0.,-0.88) -- (0.,3.1);};
 {\foreach \y in {,1.,2.,3.}\draw[shift={(0,\y)},color=grisclar] (1.1,0.) -- (-2pt,0.);};
{\clip(-0.22,-0.88) rectangle (1.1,3.1);};
{\draw (1.0,-0.1) node[anchor=north west] {$\frac{1}{p}$};};
{\draw [color=grisclar, dash pattern=on 2pt off 2pt,domain=-0.12:1.1] plot(\x,{(-0.--2.*\x)/1.});};
{\draw [color=negre, dash pattern=on 1pt off 3pt] (0.3,0) -- (0.3,1);};
{\draw [color=negre, dash pattern=on 1pt off 3pt] (0.7,0) -- (0.7,1);};
{\draw [color=grisclar] (1.,-0.88) -- (1.,3.1);};
{\draw (-0.18,3.1) node[anchor=north west] {$s$};};
{\node [circle, draw=negre, thick,inner sep=0pt, minimum size=1pt] (1) at (0.7,0) {$\,$};};
{\draw (0.43,-0.07) node[anchor=north east] {$\frac{1}{p_\kappa}$};};
{\node [circle, draw=negre, thick,inner sep=0pt, minimum size=1pt] (2) at (0.3,0) {$\,$};};
{\draw (0.83,-0.07) node[anchor=north east] {$\frac{1}{p_\kappa'}$};};
{\draw [color=grisfosc] (0.7,3) node[anchor=north west] {$W^{s,p}_{loc}$};};
{\draw (1.12,-0.1) node[anchor=north west] {$\frac{1}{p}$};};
%ClopFaracoMateuOrobitgZhongSubcritical
{\draw[color=red] (0.95,2) node[anchor=south west] {$f$};};
{\fill[color=red,fill=vermell,fill opacity=0.5] (0.,0.) -- (1.,0.) -- (1.,2) -- (0.95,2) -- (0.3,1) -- (0,0.4)-- cycle;};
{\draw[color=red] (0,0.4) -- (0.,0.) -- (1.,0.) -- (1.,2) -- (0.95,2);};
{\draw[color=red, dash pattern=on 2pt off 2pt] (0,0.4) -- (0.3,1) -- (0.95,2.) ;};
{\draw [color=blau] (0.65,1) node[anchor=south east] {$\mu$};};
{\draw[color= blau] (0.,0.) -- (1.,0.) -- (1.,1) -- (0.65,1)  -- cycle;};
{\draw [color= blauverd] (0.8,1) node[anchor=south west] {$\bar\partial f$};};
{\draw[color=blauverd, dash pattern=on 1pt off 3pt] (0.3,0) -- (0.95,1.) ;};
{\draw [color= blauverd, thick] (0.95,0.975) -- (1,0.975);};
\begin{scriptsize}
{\draw [color=blau, thick] (0.65,1)-- ++(-1.5pt,-1.5pt) -- ++(3.0pt,3.0pt) ++(-3.0pt,0) -- ++(3.0pt,-3.0pt);};
{\node [circle, draw=blauverd,inner sep=0pt, minimum size=2pt] (1) at (0.95,1) {$\,$};};
{\node [circle, draw=red,inner sep=0pt, minimum size=2pt] (1) at (0.95,2) {$\,$};};
{\node [circle, draw=red,inner sep=0pt, minimum size=2pt] (1) at (0.3,1) {$\,$};};
\end{scriptsize}
\end{tikzpicture}
}\,\,
\subfigure[From \cite{ClopFaracoMateuOrobitgZhong}, if $p=2$, $\mu\in W^{1,p} \implies \bar\partial f \in W^{1,q}$ when $q<p$.]
{
\begin{tikzpicture}[line cap=round,line join=round,>=triangle 45,x=3.0cm,y=1.12cm]
\draw[->,color=grisclar] (-0.12,0.) -- (1.1,0.);
\foreach \x in {0,0.25,0.5,0.75,1.}
\draw[shift={(\x,0)},color=grisclar] (0pt,2pt) -- (0pt,-2pt);
 { \draw[->,color=grisclar] (0.,-0.88) -- (0.,3.1);};
 {\foreach \y in {,1.,2.,3.}\draw[shift={(0,\y)},color=grisclar] (1.1,0.) -- (-2pt,0.);};
{\clip(-0.22,-0.88) rectangle (1.1,3.1);};
{\draw (1.0,-0.1) node[anchor=north west] {$\frac{1}{p}$};};
{\draw [color=grisclar, dash pattern=on 2pt off 2pt,domain=-0.12:1.1] plot(\x,{(-0.--2.*\x)/1.});};
{\draw [color=grisclar] (1.,-0.88) -- (1.,3.1);};
{\draw (-0.18,3.1) node[anchor=north west] {$s$};};
{\node [circle, draw=negre, thick,inner sep=0pt, minimum size=1pt] (1) at (0.7,0) {$\,$};};
{\draw (0.43,-0.07) node[anchor=north east] {$\frac{1}{p_\kappa}$};};
{\node [circle, draw=negre, thick,inner sep=0pt, minimum size=1pt] (2) at (0.3,0) {$\,$};};
{\draw (0.83,-0.07) node[anchor=north east] {$\frac{1}{p_\kappa'}$};};
{\draw [color=grisfosc] (0.7,3) node[anchor=north west] {$W^{s,p}_{loc}$};};
{\draw (1.12,-0.1) node[anchor=north west] {$\frac{1}{p}$};};
%ClopFaracoMateuOrobitgZhongCritical
{\draw[color=red] (0.5,2) node[anchor=south west] {$f$};};
{\fill[color=red,fill=vermell,fill opacity=0.5] (0.,0.) -- (1.,0.) -- (1.,2) -- (0.5,2) -- (0,1) -- cycle;};
{\draw[color=red] (0,1) -- (0.,0.) -- (1.,0.) -- (1.,2) -- (0.5,2);};
{\draw[color=red, dash pattern=on 2pt off 2pt] (0,1) -- (0.5,2.) ;};
{\draw [color=blau] (0.5,1) node[anchor=south east] {$\mu$};};
{\draw[color= blau] (0.,0.) -- (1.,0.) -- (1.,1) -- (0.5,1)  -- cycle;};
{\draw [color= blauverd] (0.5,1) node[anchor=south west] {$\bar\partial f$};};
{\draw [color= blauverd, thick] (0.5,0.975) -- (1,0.975);};
{\draw [color= blauverd,  dash pattern=on 2pt off 2pt, thick] (0.0125,0) -- (0.5,0.975);};
\begin{scriptsize}
{\draw [color=blau, thick] (0.5,1)-- ++(-1.5pt,-1.5pt) -- ++(3.0pt,3.0pt) ++(-3.0pt,0) -- ++(3.0pt,-3.0pt);};
{\node [circle, draw=blauverd,inner sep=0pt, minimum size=2pt] (1) at (0.5,1) {$\,$};};
{\node [circle, draw=blauverd,inner sep=0pt, minimum size=2pt] (1) at (0.5,2) {$\,$};};
\end{scriptsize}
\end{tikzpicture}
}\,\,
\subfigure[From \cite{CruzMateuOrobitg}, if $sp>2$, $\mu\in W^{s,p} \implies \bar\partial f \in W^{s,p}$.]
{
\begin{tikzpicture}[line cap=round,line join=round,>=triangle 45,x=3.0cm,y=1.12cm]
\draw[->,color=grisclar] (-0.12,0.) -- (1.1,0.);
\foreach \x in {0,0.25,0.5,0.75,1.}
\draw[shift={(\x,0)},color=grisclar] (0pt,2pt) -- (0pt,-2pt);
 { \draw[->,color=grisclar] (0.,-0.88) -- (0.,3.1);};
 {\foreach \y in {,1.,2.,3.}\draw[shift={(0,\y)},color=grisclar] (1.1,0.) -- (-2pt,0.);};
{\clip(-0.22,-0.88) rectangle (1.1,3.1);};
{\draw (1.0,-0.1) node[anchor=north west] {$\frac{1}{p}$};};
{\draw [color=grisclar, dash pattern=on 2pt off 2pt,domain=-0.12:1.1] plot(\x,{(-0.--2.*\x)/1.});};
{\draw [color=grisclar] (1.,-0.88) -- (1.,3.1);};
{\draw (-0.18,3.1) node[anchor=north west] {$s$};};
{\node [circle, draw=negre, thick,inner sep=0pt, minimum size=1pt] (1) at (0.7,0) {$\,$};};
{\draw (0.43,-0.07) node[anchor=north east] {$\frac{1}{p_\kappa}$};};
{\node [circle, draw=negre, thick,inner sep=0pt, minimum size=1pt] (2) at (0.3,0) {$\,$};};
{\draw (0.83,-0.07) node[anchor=north east] {$\frac{1}{p_\kappa'}$};};
{\draw [color=grisfosc] (0.7,3) node[anchor=north west] {$W^{s,p}_{loc}$};};
{\draw (1.12,-0.1) node[anchor=north west] {$\frac{1}{p}$};};
{\draw[color=red] (0.5,2.5) node[anchor=south west] {$f$};};
{\fill[color=red,fill=vermell,fill opacity=0.5] (0.,0.) -- (1.,0.) -- (1.,2.5) -- (0.5,2.5) -- (0,1.5) -- cycle;};
{\draw[color=red] (0.,0.) -- (1.,0.) -- (1.,2.5) -- (0.5,2.5) -- (0,1.5) -- cycle;};
{\draw [color=blau] (0.5,1.5) node[anchor=south east] {$\mu$};};
{\draw[color= blau] (0.,0.) -- (1.,0.) -- (1.,1.5) -- (0.5,1.5) -- (0,0.5) -- cycle;};
{\draw [color= blauverd] (0.5,1.5) node[anchor=south west] {$\bar\partial f$};};
{\draw [color= blauverd, thick] (0.0075,0) -- (0.0075,0.49) -- (0.5025,1.48) -- (1,1.48);};
%{\draw[color= blauverd] (0.,0.) -- (1.,0.) -- (1.,1.5) -- (0.5,1.5) -- (0,0.5) -- cycle;};
\begin{scriptsize}
{\draw [color=blau, thick] (0.5,1.5)-- ++(-1.5pt,-1.5pt) -- ++(3.0pt,3.0pt) ++(-3.0pt,0) -- ++(3.0pt,-3.0pt);};
{\draw [color=blauverd]  (0.5,1.5)-- ++(-1pt,-1pt) -- ++(2.0pt,2.0pt) ++(-2.0pt,0) -- ++(2.0pt,-2.0pt);};
{\draw [color=red] (0.5,2.5)-- ++(-1.5pt,-1.5pt) -- ++(3.0pt,3.0pt) ++(-3.0pt,0) -- ++(3.0pt,-3.0pt);};
\end{scriptsize}
\end{tikzpicture}
}
\caption{Regularity of the principal quasiconformal solution to \rf{eqBeltrami} when the coefficients satisfy a-priori conditions.}\label{figDiscussionResults}
\end{figure}

Kari Astala showed in \cite{Astala} that every quasiconformal mapping $f$ with Beltrami coefficient $\mu\in L^\infty$ compactly supported in the unit disk with $\norm{\mu}_{L^\infty}=\kappa<1$ satisfies that
$$\bar{\partial} f \in L^p \mbox{ \quad\quad whenever } \frac{1}{p_\kappa} < \frac1p \leq 1, $$
(see Figure \ref{figDiscussionResults}(a)). Moreover, the operator $\mathcal{H}=Id-\mu\Beurling-\nu\overline{\Beurling}$ is invertible in Lebesgue spaces with exponent on the critical range $(p_\kappa', p_{\kappa})$ as shown in \cite{AstalaIwaniecSaksman}. When the regularity of the coefficients is greater, we can expect the principal solution to \rf{eqBeltrami} to have greater regularity as well. The first result in that direction was given by Iwaniec, who could prove the compactness of the commutator $[\mu,\Beurling]$ in every $L^p$ space when $\mu$ is a Beltrami coefficient in $VMO$. By means of a Fredholm theory argument this implies the invertibility of $\mathcal{H}$ in $L^p$ when $\nu = 0$, and the general case is shown by the same argument (see Figure \ref{figDiscussionResults}(b)).

In recent years there has been a great improvement in results. A remarkable one, given by Clop et al. in \cite[Proposition 4]{ClopFaracoMateuOrobitgZhong}, deals with the case of $f$ being the principal solution to the Beltrami equation \rf{eqBeltrami} with $\mu\in W^{1,p}_{c}$ with  $1<p<\infty$ and $\nu\equiv 0$ (see Figure \ref{figDiscussionResults} (d), (e) and (f)):
\begin{itemize}
\item If $\frac1p<\frac12$, then $ \bar{\partial}f \in W^{1,p}$.
\item If $\frac1p=\frac12$, then $ \bar{\partial}f \in W^{1,q}$ for every $q<2$. 
\item If $\frac12<\frac1p<\frac1{p_\kappa'}$, then $\bar{\partial}f \in W^{1,q}$ for every $\frac{1}{q} > \frac{1}{p}+\frac{1}{p_\kappa}$. 
\end{itemize}
Notice how $W^{1,p}$ being either supercritical, critical or subcritical determines the regularity behavior of $f$. In the subcritical case, the setting $\nu\neq 0$ was studied in \cite[Corolario 3.8]{BaisonThesis}. The author shows that when $\frac12<\frac1p<\frac1{p_\kappa'}-\frac1{p_\kappa}$, then $\bar{\partial}f \in W^{1,q}$ for every $\frac{1}{q} > \frac{1}{p}+\frac{1}{p_\kappa}$. Here we observe the same phenomenon as in Theorem \ref{theoRegularityCoefficients}, that is, we need that $\kappa<\frac13$ in order to have room for $p$ in the hypothesis. 

In the supercritical case there is no loss of regularity: if $\mu \in X$, then $\bar\partial f \in X$ as well. Indeed this is the case in the H\"older spaces of fractional order (see \cite[Chapter 15]{AstalaIwaniecMartin},  Figure \ref{figDiscussionResults}(c) above) and in the whole Triebel-Lizorkin and Besov scales, as shown by Cruz, Mateu and Orobitg, see \cite[Theorem 1]{CruzMateuOrobitg},  Figure \ref{figDiscussionResults}(f). Here, the authors had to prove the invertibility of $I-\mu \Beurling$ on $F^{s}_{p,q}$ which they did using Fredholm theory following Iwaniec's scheme, since the boundedness of $\Beurling$ on Triebel-Lizorkin spaces can be deduced from \cite[Corollary 3.33]{FrazierTorresWeiss}, while in Besov spaces it is a consequence of Fourier multipliers theory and interpolation (see \cite[Section 2.6]{TriebelTheory}, for instance). In the supercritical context but restricted to domains, $\mu\in W^{s,p}(\Omega)$, there are also some positive results in \cite{CruzMateuOrobitg} and in \cite{PratsQuasiconformal}.

When $\mu$ belongs to a critical space  $X\subset VMO$, we cannot expect that $\bar\partial f$ is in $X$, but the loss is minimal. The first result in that spirit is Iwaniec's above.  This theorem has been extended to other critical spaces in \cite{ClopFaracoMateuOrobitgZhong} for $s=1$ as commented above, and in \cite[Theorem 1]{BaisonClopOrobitg} which settles the case $\frac12 \leq s<1$, finding that 
$$ \mu \in W^{s,2/s} \implies\log(\partial f) \in W^{s,2/s}.$$
This result implies \rf{eqBigTarget2} for $\frac12 \leq s<1$. Thus, the progress in Theorem \ref{theoRegularityCoefficients} for the critical setting is to cover the whole range $0<s<2$.
Finally, small improvements on the spaces lead to better results, as shown in \cite{CruzMateuOrobitg}, where the authors prove that replacing $W^{s,2/s}$ by the Riesz potential of a Lorentz space $I^s(L^{2/s,1})$, then there is no loss. In this case, the functions are continuous, so we can classify these spaces as supercritical.

In the subcritical situation, in addition to the results on classical spaces in \cite[Corollary 1.2]{Astala}, \cite[Proposition 4]{ClopFaracoMateuOrobitgZhong} and \cite[Corolario 3.8]{BaisonThesis}, the fractional smoothness case $0<s<1$ was considered in \cite[Theorem 4.3]{ClopFaracoRuiz}. In this case, the authors show that $\bar{\partial} f\in W^{\Theta s,2}$ for every $\Theta < 1-\frac{2}{p_\kappa}$. Note that in this result there is a  loss in ``smoothness'', that is, in the $s$ parameter. We will show Theorem \ref{theoRegularityCoefficients} using the same techniques as the authors of that result with some extra care to avoid this loss. We recover their result in Corollary \ref{coroRegularityCoefficients}.

However, when $s=1$ and $\nu=0$ \cite[Proposition 4]{ClopFaracoMateuOrobitgZhong} is stronger than Theorem \ref{theoRegularityCoefficients}, since the restriction $\frac{1}{p}<\frac{1}{p_\kappa'}-\frac{1}{p_\kappa}$ is replaced by $\frac{1}{p}<\frac{1}{p_\kappa'}$ (compare Figures \ref{figDiscussionResults}(d) and \ref{figComparativeSubcritical}(c)). Thus, it is natural to ask whether the following conjecture holds.

\begin{conjecture}\label{conjRegularityCoefficients}
Let $0<s<2$, $1<p<\infty$, with $s<\frac2p$, let $\mu,\nu \in F^s_{p,q} \cap L^\infty_{c}$  satisfy \rf{eqElliptic} and let $f$ be the principal solution to the Beltrami equation \rf{eqBeltrami}. If $\frac{1}{p}<\frac{1}{p_\kappa'}$, then
 $$\bar{\partial} f\in F^s_{p,q} \mbox{ \quad\quad for every } \frac{1}{q} > \frac{1}{p}+\frac{1}{p_\kappa}.$$
\end{conjecture}

\begin{figure}[ht]
\centering
\subfigure[Behavior for $s<1$ and $p=2$ described in {\cite[Theorem 4.3]{ClopFaracoRuiz}}.]
{
\begin{tikzpicture}[line cap=round,line join=round,>=triangle 45,x=3.0cm,y=1.12cm]
\draw[->,color=grisclar] (-0.12,0.) -- (1.1,0.);
\foreach \x in {0,0.25,0.5,0.75,1.}
\draw[shift={(\x,0)},color=grisclar] (0pt,2pt) -- (0pt,-2pt);
 { \draw[->,color=grisclar] (0.,-0.88) -- (0.,3.1);};
 {\foreach \y in {,1.,2.,3.}\draw[shift={(0,\y)},color=grisclar] (1.1,0.) -- (-2pt,0.);};
{\clip(-0.22,-0.88) rectangle (1.1,3.1);};
{\draw (1.0,-0.1) node[anchor=north west] {$\frac{1}{p}$};};
{\draw [color=grisclar, dash pattern=on 2pt off 2pt,domain=-0.12:1.1] plot(\x,{(-0.--2.*\x)/1.});};
{\draw [color=grisclar] (1.,-0.88) -- (1.,3.1);};
{\draw [color=grisclar, dash pattern=on 1pt off 3pt] (0.5,0) -- (0.5,1.24);};
{\draw [color=negre, dash pattern=on 1pt off 3pt] (0.3,0) -- (0.3,1);};
{\draw (-0.18,3.1) node[anchor=north west] {$s$};};
{\node [circle, draw=negre, thick,inner sep=0pt, minimum size=1pt] (1) at (0.7,0) {$\,$};};
{\draw (0.43,-0.07) node[anchor=north east] {$\frac{1}{p_\kappa}$};};
{\node [circle, draw=negre, thick,inner sep=0pt, minimum size=1pt] (2) at (0.3,0) {$\,$};};
{\draw (0.83,-0.07) node[anchor=north east] {$\frac{1}{p_\kappa'}$};};
{\draw [color=grisfosc] (0.7,3) node[anchor=north west] {$W^{s,p}_{loc}$};};
{\draw (1.12,-0.1) node[anchor=north west] {$\frac{1}{p}$};};
%ClopFaracoMateuOrobitgZhongSubcritical
{\draw[color=red] (0.5,1.6) node[anchor=south west] {$f$};};
{\fill[color=red,fill=vermell,fill opacity=0.5] (0.,0.) -- (1.,0.) -- (1.,1.24) -- (0.5,1.24) -- (0.3,1) -- (0,0.4)-- cycle;};
{\draw[color=red] (0,0.4) -- (0.,0.) -- (1.,0.) -- (1,1.24) ;};
{\draw[color=red, dash pattern=on 2pt off 2pt] (0,0.4) -- (0.3,1) -- (0.5,1.24)  -- (1,1.24);};
{\draw [color=blau] (0.5,0.6) node[anchor=south west] {$\mu$};};
{\draw[color= blau] (0.,0.) -- (1.,0.) -- (1.,0.6) -- (0.5,0.6)  -- cycle;};
{\draw [color= blauverd] (0.5,0.14) node[anchor=south west] {$\bar\partial f$};};
%{\draw[color= blauverd] (0.3,0.) -- (1.,0.) -- (1.,1) -- (0.95,1) ;};
{\draw[color=blauverd, dash pattern=on 1pt off 3pt] (0.3,0) -- (0.5,0.24) -- (1,0.24) ;};
\begin{scriptsize}
{\draw [color=blau, thick] (0.5,0.6)-- ++(-1.5pt,-1.5pt) -- ++(3.0pt,3.0pt) ++(-3.0pt,0) -- ++(3.0pt,-3.0pt);};
{\node [circle, draw=blauverd,inner sep=0pt, minimum size=2pt] (1) at (0.5,0.24) {$\,$};};
{\node [circle, draw=red,inner sep=0pt, minimum size=2pt] (1) at (0.5,1.24) {$\,$};};
{\node [circle, draw=red,inner sep=0pt, minimum size=2pt] (1) at (0.3,1) {$\,$};};
\end{scriptsize}
\end{tikzpicture}
}\,\,
\subfigure[Behavior described in Theorem \ref{theoRegularityCoefficients}.]
{
\begin{tikzpicture}[line cap=round,line join=round,>=triangle 45,x=3.0cm,y=1.12cm]
\draw[->,color=grisclar] (-0.12,0.) -- (1.1,0.);
\foreach \x in {0,0.25,0.5,0.75,1.}
\draw[shift={(\x,0)},color=grisclar] (0pt,2pt) -- (0pt,-2pt);
 { \draw[->,color=grisclar] (0.,-0.88) -- (0.,3.1);};
 {\foreach \y in {,1.,2.,3.}\draw[shift={(0,\y)},color=grisclar] (1.1,0.) -- (-2pt,0.);};
{\clip(-0.22,-0.88) rectangle (1.1,3.1);};
{\draw (1.0,-0.1) node[anchor=north west] {$\frac{1}{p}$};};
{\draw [color=grisclar, dash pattern=on 2pt off 2pt,domain=-0.12:1.1] plot(\x,{(-0.--2.*\x)/1.});};
{\draw [color=grisclar] (1.,-0.88) -- (1.,3.1);};
{\draw [color=negre, dash pattern=on 1pt off 3pt] (0.3,0) -- (0.3,1);};
{\draw [color=negre, dash pattern=on 1pt off 3pt] (0.7,0) -- (0.7,1);};
{\draw (-0.18,3.1) node[anchor=north west] {$s$};};
{\node [circle, draw=negre, thick,inner sep=0pt, minimum size=1pt] (1) at (0.7,0) {$\,$};};
{\draw (0.43,-0.07) node[anchor=north east] {$\frac{1}{p_\kappa}$};};
{\node [circle, draw=negre, thick,inner sep=0pt, minimum size=1pt] (2) at (0.3,0) {$\,$};};
{\draw (0.83,-0.07) node[anchor=north east] {$\frac{1}{p_\kappa'}$};};
{\draw [color=grisfosc] (0.7,3) node[anchor=north west] {$W^{s,p}_{loc}$};};
{\draw (1.12,-0.1) node[anchor=north west] {$\frac{1}{p}$};};
%New results
{\draw[color=red] (0.65,1.42) node[anchor=south west] {$f$};};
{\fill[color=red,fill=vermell,fill opacity=0.5] (0.,0.) -- (1.,0.) -- (1.,1.42) -- (0.65,1.42) -- (0.3,1) -- (0,0.4)-- cycle;};
{\draw[color=red] (0,0.4) -- (0.,0.) -- (1.,0.) -- (1,1.42) -- (0.65,1.42);};
{\draw[color=red, dash pattern=on 2pt off 2pt] (0,0.4) -- (0.3,1) -- (0.65,1.42);};
{\draw [color=blau] (0.35,0.42) node[anchor=south west] {$\mu$};};
{\draw[color= blau] (0.,0.) -- (1.,0.) -- (1.,0.42) -- (0.35,0.42)  -- cycle;};
{\draw [color= blauverd] (0.65,0.42) node[anchor=south west] {$\bar\partial f$};};
%{\draw[color= blauverd] (0.3,0.) -- (1.,0.) -- (1.,1) -- (0.95,1) ;};
{\draw[color=blauverd, dash pattern=on 2pt off 2pt] (0.3,0) -- (0.65,0.42);};
{\draw[color=blauverd, thick] (0.65,0.4) -- (1, 0.4) ;};
{\draw [color=negre, dash pattern=on 1pt off 3pt] (0.65,0.4) -- (0.65,1.4);};
\begin{scriptsize}
{\draw [color=blau, thick] (0.35,0.42)-- ++(-1.5pt,-1.5pt) -- ++(3.0pt,3.0pt) ++(-3.0pt,0) -- ++(3.0pt,-3.0pt);};
{\node [circle, fill=white, inner sep=0pt, minimum size=2pt] (1) at (0.65,0.42) {$\,$};};
{\node [circle, draw=blauverd,fill=vermell,fill opacity=0.5, inner sep=0pt, minimum size=2pt] (1) at (0.65,0.42) {$\,$};};
{\node [circle, draw=red,inner sep=0pt, minimum size=2pt] (1) at (0.65,1.42) {$\,$};};
{\node [circle, draw=red,inner sep=0pt, minimum size=2pt] (1) at (0.3,1) {$\,$};};
\end{scriptsize}
\end{tikzpicture}
}\,\,
\subfigure[Behavior described in Corollary \ref{coroRegularityCoefficients}.]
{
\begin{tikzpicture}[line cap=round,line join=round,>=triangle 45,x=3.0cm,y=1.12cm]
\draw[->,color=grisclar] (-0.12,0.) -- (1.1,0.);
\foreach \x in {0,0.25,0.5,0.75,1.}
\draw[shift={(\x,0)},color=grisclar] (0pt,2pt) -- (0pt,-2pt);
 { \draw[->,color=grisclar] (0.,-0.88) -- (0.,3.1);};
 {\foreach \y in {,1.,2.,3.}\draw[shift={(0,\y)},color=grisclar] (1.1,0.) -- (-2pt,0.);};
{\clip(-0.22,-0.88) rectangle (1.1,3.1);};
{\draw (1.0,-0.1) node[anchor=north west] {$\frac{1}{p}$};};
{\draw [color=grisclar, dash pattern=on 2pt off 2pt,domain=-0.12:1.1] plot(\x,{(-0.--2.*\x)/1.});};
{\draw [color=grisclar] (1.,-0.88) -- (1.,3.1);};
{\draw [color=negre, dash pattern=on 1pt off 3pt] (0.3,0) -- (0.3,1);};
{\draw [color=negre, dash pattern=on 1pt off 3pt] (0.7,0) -- (0.7,1.54);};
{\draw (-0.18,3.1) node[anchor=north west] {$s$};};
{\node [circle, draw=negre, thick,inner sep=0pt, minimum size=1pt] (1) at (0.7,0) {$\,$};};
{\draw (0.43,-0.07) node[anchor=north east] {$\frac{1}{p_\kappa}$};};
{\node [circle, draw=negre, thick,inner sep=0pt, minimum size=1pt] (2) at (0.3,0) {$\,$};};
{\draw (0.83,-0.07) node[anchor=north east] {$\frac{1}{p_\kappa'}$};};
{\draw [color=grisfosc] (0.7,3) node[anchor=north west] {$W^{s,p}_{loc}$};};
{\draw (1.12,-0.1) node[anchor=north west] {$\frac{1}{p}$};};
%New results
{\draw[color=red] (0.7,1.54) node[anchor=south west] {$f$};};
{\fill[color=red,fill=vermell,fill opacity=0.5] (0.,0.) -- (1.,0.) -- (1.,1.54) -- (0.7,1.54) -- (0.3,1) -- (0,0.4)-- cycle;};
{\draw[color=red] (0,0.4) -- (0.,0.) -- (1.,0.) -- (1,1.54);};
{\draw[color=red, dash pattern=on 2pt off 2pt] (0,0.4) -- (0.3,1) -- (0.7,1.54) -- (1,1.54);};
{\draw [color=blau] (0.6,0.65) node[anchor=south east] {$\mu$};};
{\draw[color= blau] (0.,0.) -- (1.,0.) -- (1.,0.81) -- (0.6,0.81)  -- cycle;};
{\draw [color= blauverd] (0.7,0.35) node[anchor=south west] {$\bar\partial f$};};
%{\draw[color= blauverd] (0.3,0.) -- (1.,0.) -- (1.,1) -- (0.95,1) ;};
{\draw[color=blauverd, dash pattern=on 1pt off 3pt] (0.3,0) -- (0.7,0.54) -- (1,0.54);};
\begin{scriptsize}
{\draw [color=blau, thick] (0.6,0.81)-- ++(-1.5pt,-1.5pt) -- ++(3.0pt,3.0pt) ++(-3.0pt,0) -- ++(3.0pt,-3.0pt);};
{\node [circle, fill=white, inner sep=0pt, minimum size=2pt] (1) at (0.7,0.54) {$\,$};};
{\node [circle, draw=blauverd,fill=vermell,fill opacity=0.5, inner sep=0pt, minimum size=2pt] (1) at (0.7,0.54) {$\,$};};
{\node [circle, draw=red,inner sep=0pt, minimum size=2pt] (1) at (0.7,1.54) {$\,$};};
{\node [circle, draw=red,inner sep=0pt, minimum size=2pt] (1) at (0.3,1) {$\,$};};
\end{scriptsize}
\end{tikzpicture}
}
\caption{Regularity of the principal quasiconformal solution to \rf{eqBeltrami} when the coefficients satisfy a-priori conditions.}\label{figComparativeSubcritical}
\end{figure}
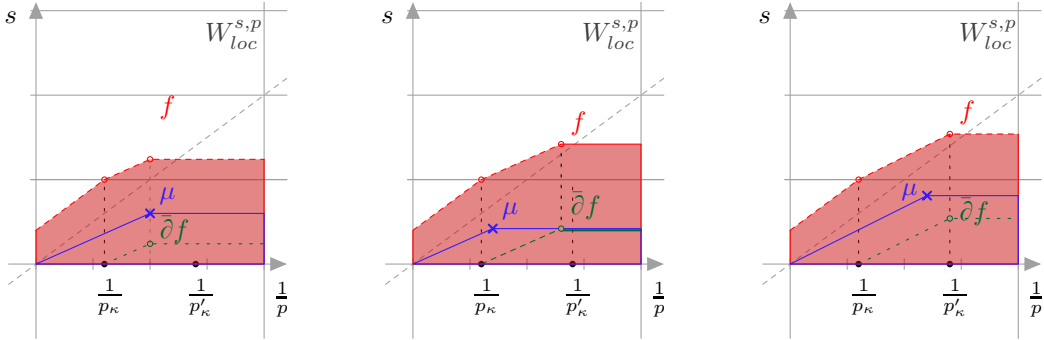

\section{Proof of Theorem \ref{theoRegularityCoefficients}}\label{secProof}
\begin{definition}
Given   $s\in\R$ and a tempered distribution $f\in \mathcal{S}' (\C)$, we define its fractional derivative 
$$D^s f := \left(|\cdot|^{s} \widehat{f}(\cdot)\right)\widecheck{\,}$$
in the sense of tempered distributions modulo polynomials (see \cite[Sections 5.1.2, 5.2.3]{TriebelTheory}).
\end{definition}
%!!!!!!!!!! Better with Cabre definition?

The Beurling transform commutes with fractional differentiation of order smaller or equal than $s$ on $W^{s,p}$. Indeed, since the Beurling transform is bounded in $W^{s,p}$ (see \cite[Section 2.6.6]{TriebelTheory}, for instance), we have that both $D^s \circ \Beurling$ and $\Beurling \circ D^s$ map $W^{s,p}$ into $L^p$. In the $\mathcal{S}$ class, which is dense, we have that $ D^s \Beurling f= \left(|\xi|^{s} \frac{\bar{\xi}}{\xi}\widehat{f}(\cdot)\right)\widecheck{\,} = \Beurling D^s f$, so the equality extends to $W^{s,p}$. 

Next we define the commutator of the multiplication by a test function with the fractional differentiation.
\begin{definition} 
For every pair $\mu , f \in C^\infty_c$, we define
\begin{equation*}
[\mu, D^s] f := \mu D^s f -  D^s (\mu f ).
\end{equation*}
\end{definition}
Next we recover Kato-Ponce Leibniz' rule as presented by Hofmann.
\begin{theorem}[see {\cite[Corollary 1.2]{Hofmann}}]\label{theoHofmann}
Let $0<s<1$, $1<r<\infty$ and  $\mu\in C^\infty_c$. The commutator $[\mu, D^s]$ extends to a linear operator with kernel
\begin{equation*}
K^s_\mu(x,y)=c\frac{\mu(x)-\mu(y)}{|x-y|^s},
\end{equation*}
mapping $L^p$ to $L^q$ for every $1\leq q < p\leq \infty$, whenever $\frac{1}{r}:=\frac{1}{q}-\frac{1}{p} \in (0,1)$. Moreover, 
$$\norm{[\mu, D^s] f}_{L^q}\lesssim_{p,q} \norm{D^s \mu}_{L^r}\norm{f}_{L^p}.$$
\end{theorem}

We will also use  the Young inequality. It states that for measurable functions $f$ and $g$, we have that
\begin{equation}\label{eqYoung}
\norm{f * g}_{L^q}\leq\norm{f}_{L^r}\norm{g}_{L^p}
\end{equation}
for $1\leq p,q,r\leq \infty$ with $\frac{1}{q}=\frac{1}{p}+\frac{1}{r}-1$ (see \cite[Appendix A2]{SteinPetit}).

\begin{proof}[Proof of Theorem \ref{theoRegularityCoefficients}]
We will study in one stroke the critical and the subcritical case. To do so, we will use the  convention $p_\kappa = \infty$ in the critical situation, that is, whenever $sp=2$, while we use the standard notation $p_\kappa=1+\frac1\kappa$ in the subcritical situation ($sp<2$). Thus, $I-\mu\Beurling$ is invertible in $L^q$ for $p_\kappa'<q<p_\kappa$ (see Theorem \ref{theoAIS} for the subcritical setting and \cite[Section 1]{Iwaniec} for the critical one).

Without loss of generality, we may assume that $\supp (\mu) \cup \supp (\nu)\subset \DDD$. Let $\psi\in C^\infty_c(\DDD)$ be a positive radial function such that $\int \psi = 1$ and consider the approximation of the unity $\psi_n(z) := n^2 \psi (n z)$ for every $n\in\N$. Consider the families of functions $\mu_n:=\psi_n * \mu$ and $\nu_n:=\psi_n * \nu$, which are supported in the disk $2\DDD$. If $\nu_n(z) \mu_n(z)\neq 0$, then
$$|\mu_n(z)|+|\nu_n(z)| = |\mu_n(z)|\left|\frac{|\nu_n(z)|}{ |\mu_n(z)|} + 1\right| = \left|\frac{\mu_n(z)}{ |\mu_n(z)|}\frac{|\nu_n(z)|}{\nu_n(z)}\nu_n(z) + \mu_n(z)\right|= \left| (c_z \nu + \mu) * \psi_n(z)\right|,$$
where $c_z$ stands for the unitary complex number $\frac{\mu_n(z)}{ |\mu_n(z)|}\frac{|\nu_n(z)|}{\nu_n(z)}$. This expression, together with the hypothesis \rf{eqElliptic} and the Young inequality \rf{eqYoung}, imply that
$$|\mu_n(z)|+|\nu_n(z)| \leq \kappa.$$

On the other hand, being the convolution with $\psi_n$ an approximation of the identity, $D^s \mu_n=D^s\mu * \psi_n$ converges to $D^s \mu$ in $L^p$ if $D^s \mu\in L^p$ for $p<\infty$, and the same can be said about $\{\nu_n\}$. Thus, 
\begin{equation}\label{eqMunUniformControl}
\norm{\mu_n-\mu}_{W^{s,p}}\xrightarrow {n\to \infty} 0 \mbox{\quad\quad and \quad\quad} \norm{\nu_n-\nu}_{W^{ s,p}}\xrightarrow {n\to \infty} 0.
\end{equation}
Let us define 
\begin{equation}\label{eqDefiningH}
h:= (I-\mu \Beurling - \nu \bar{\Beurling})^{-1}(\mu+\nu)
\end{equation}
and
\begin{equation}\label{eqDefiningHn}
h_n:= (I-\mu_n\Beurling - \nu_n\bar{\Beurling})^{-1}(\mu_n+\nu_n).
\end{equation}
 Then, $f_n(z):=z+\Cauchy h_n(z)$ is the principal solution to \rf{eqBeltrami} with coefficients $\mu_n$ and $\nu_n$ (\cite[Lemma 4.2]{ClopFaracoRuiz}), and we have that 
\begin{equation}\label{eqhnUniformControl}
\norm{h_n}_{L^\ell}\leq C_{\kappa,\ell}
\end{equation}
for every $\frac{1}{\ell}>\frac1{p_\kappa}$.

The case $s=1$ is fully covered by \cite[Proposition 4]{ClopFaracoMateuOrobitgZhong} and \cite[Corolario 3.8]{BaisonThesis}.  Let us consider first the case $s<1$. 
In terms of $h_n$, equation \rf{eqBeltrami} read as
$$h_n = \mu_n \Beurling h_n +\mu_n + \nu_n \overline{\Beurling} h_n + \nu_n.$$
Note that $h_n$ is $C^\infty_c$, and $\mu_n, \nu_n \in C^\infty_c$ as well, while $\Beurling h_n\in C^\infty \cap W^{m,p}$ for every $m\in \N$ and $1<p<\infty$. Thus, we can take fractional derivatives of order $s$ to get
\begin{align*}
D^s h_n 
\nonumber	& = D^s(\mu_n \Beurling h_n  + \nu_n \overline{\Beurling} h_n+\mu_n + \nu_n) \\
	& = \mu_n D^s \Beurling h_n  -[\mu_n, D^s](\Beurling h_n) + \nu_n D^s \overline{\Beurling} h_n -[\nu_n, D^s](\overline{\Beurling} h_n)  + D^s \mu_n + D^s \nu_n,
\end{align*}
that is, since the Beurling transform commutes with the fractional derivatives for $C^\infty_c$ functions, 
\begin{equation}\label{eqIdentityBeforeArmaggedon}
D^s h_n  - \mu_n  \Beurling (D^s h_n)  - \nu_n  \overline{\Beurling} (D^sh_n)
	 = -[\mu_n, D^s](\Beurling h_n) -[\nu_n, D^s](\overline{\Beurling} h_n) + D^s \mu_n + D^s \nu_n.
\end{equation}
Next, by \rf{eqhnUniformControl} we have that
\begin{equation}\label{eqBhnUniformControl}
\norm{\Beurling h_n}_{L^\ell}\leq C_{\kappa, q} \mbox{\quad\quad for every } \frac{1}{\ell}>\frac{1}{p_\kappa}
\end{equation}
and, by \rf{eqMunUniformControl}, for $n$ big enough
\begin{equation}\label{eqDsMunUniformControl}
\norm{D^s \mu_n}_{L^r}+\norm{D^s \nu_n}_{L^r} \leq C_{\kappa, p, r, \norm{\mu}_{W^{s,p}}, \norm{\nu}_{W^{s,p}}} \mbox{\quad\quad for every } \frac{1}{r}\geq \frac{1}{p}.
\end{equation}
Given $\frac{1}{q}>\frac{1}{p}+\frac{1}{p_\kappa}$, we can write $\frac{1}{q}=\frac{1}{\ell}+\frac{1}{r}$ satisfying restrictions \rf{eqBhnUniformControl} and \rf{eqDsMunUniformControl}, so Theorem \ref{theoHofmann} with \rf{eqhnUniformControl} and \rf{eqBhnUniformControl} yields
\begin{align}\label{eqTsUniformControl}
\norm{[\mu_n, D^s] \Beurling h_n+[\nu_n, D^s] \overline{\Beurling} h_n}_{L^q}
\nonumber	&  \lesssim_{\kappa, p, q} \left(\norm{D^s \mu_n}_{L^r}\norm{\Beurling h_n}_{L^\ell} +\norm{D^s \nu_n}_{L^r}\norm{\overline{\Beurling} h_n}_{L^\ell} \right)\\
	& \leq C_{\kappa, p, q, \norm{\mu}_{W^{s,p}}, \norm{\nu}_{W^{s,p}}}.
\end{align}
Thus, we have a uniform control on the $L^q$ norm of the right hand side of \rf{eqIdentityBeforeArmaggedon}. In case 
\begin{equation*}
\frac{1}{p}+\frac{1}{p_\kappa}<\frac{1}{p_\kappa'}
\end{equation*}
we can find  $\frac{1}{q}$ in between where we can invert the operator $I-\mu_n\Beurling -\nu_n \bar\Beurling$ and, using \rf{eqIdentityBeforeArmaggedon}, we can  write
\begin{align*}
\norm{D^s h_n}_{L^q}
	& = \norm{(I-\mu_n\Beurling - \nu_n\bar{\Beurling})^{-1} ([\mu_n, D^s](\Beurling h_n) +[\nu_n, D^s](\overline{\Beurling} h_n) - D^s \mu_n - D^s \nu_n)}_{L^q}\\
	& \lesssim \norm{ [\mu_n, D^s](\Beurling h_n) +[\nu_n, D^s](\overline{\Beurling} h_n) - D^s \mu_n - D^s \nu_n}_{L^q}
\end{align*}
which is uniformly bounded by \rf{eqInverseBeltramiOperatorUniformControl}, \rf{eqTsUniformControl} and \rf{eqDsMunUniformControl}. Combining these estimates with \rf{eqhnUniformControl}, we get 
\begin{equation}\label{eqUniformControlInH}
\norm{h_n}_{W^{s,q}}\lesssim C_{\kappa, p, q, \norm{\mu}_{W^{s,p}}, \norm{\nu}_{W^{s,p}}}.
\end{equation}
The Banach-Alaoglu Theorem shows that there exists a subsequence $h_{n_k}$ converging to $\widetilde h \in W^{s,p}$ as tempered distributions, with $\|\widetilde{h}\|_{W^{s,q}}\lesssim C_{\kappa, p, q, \norm{\mu}_{W^{s,p}}, \norm{\nu}_{W^{s,p}}}$.  

It only remains to transfer this information to $h$. First, note that 
\begin{equation}\label{eqL2Convergence}
h_n\xrightarrow {n\to\infty} h \mbox{\quad\quad in } L^2.
\end{equation}
Indeed, from \rf{eqDefiningH} and \rf{eqDefiningHn}, we have that
\begin{align*}
|h-h_n|	
	& =|\mu \Beurling h -\mu_n\Beurling h_n + \nu \bar{\Beurling} h-\nu_n \bar{\Beurling} h_n + \mu-\mu_n+\nu-\nu_n|\\
	& \leq |\mu| |\Beurling (h-h_n)| +|\mu-\mu_n||\Beurling h_n |+ |\nu| |{\Beurling} (h-h_n)| +|\nu-\nu_n| |\Beurling h_n| + |\mu-\mu_n|+|\nu-\nu_n|.
\end{align*}
Taking $L^2$ norms, 
\begin{align*}
\norm{h-h_n}_{L^2}
	& \leq \kappa \norm{\Beurling (h-h_n)}_{L^2} +\norm{(\mu-\mu_n) \Beurling h_n}_{L^2} + \norm{(\nu-\nu_n)\Beurling h_n}_{L^2} + \norm{\mu-\mu_n}_{L^2}+\norm{\nu-\nu_n}_{L^2}.
\end{align*}
Since $\kappa<1$ and the Beurling transform is an isometry in $L^2$,  the left-hand side can  absorb the first term in the right-hand side. To deal with the second and the third terms, choose $2<\widetilde{q}<p_\kappa$, and let $\frac1r+\frac1{\widetilde{q}}=\frac12$. Then, by \rf{eqDefiningHn} and Theorem \ref{theoAIS}, the norm of $h_n$ in $L^{\widetilde{q}}$ is uniformly bounded, and using the $L^p$ version of \rf{eqMunUniformControl} (i.e., Young's inequality) as well we get 
$$\norm{(\mu-\mu_n) \Beurling h_n}_{L^2} \leq \norm{\mu-\mu_n}_{L^r} \norm{\Beurling h_n}_{L^{\widetilde{q}}}\leq C_{\kappa,\widetilde{q},n} \xrightarrow[n\to\infty]{}  0 . $$
Using again Young's inequality for the remaining terms, we see that
\begin{align*}
\norm{h-h_n}_{L^2}
	& \leq \frac{C_{\kappa,n}}{1-\kappa}  \xrightarrow[n\to\infty]{}  0 ,
\end{align*}
showing \rf{eqL2Convergence}. Since this implies convergence as tempered distributions and $\mathcal{S}'$ is a Hausdorff space, we get that $h=\widetilde{h}$. This finishes the case $s<1$. 

It remains to show that \rf{eqBigTarget} holds when $\frac{1}{p}<\frac{1}{p_\kappa'}-\frac{1}{p_\kappa}$ and $s>1$. If this is the case, then for every $1<q<\infty$ we have that $\norm{h}_{ W^{s,q}}\approx \norm{h}_{L^q} + \norm{\partial h}_{W^{s-1,q}}$ by the lifting property (see \cite[Section 2.3.8]{TriebelTheory}) and the boundedness of the Beurling transform in Sobolev spaces.
We will assume that $\nu=\nu_n=0$ to keep a compact notation, leaving the necessary modifications to the reader. Fix $n\geq 1$. By assumption, $\mu \in W^{s,p}\cap L^\infty$. By \rf{eqInterpolateLInfty} we also have that 
\begin{equation}\label{eqUniformControlDerivativesMuN}
\mu, \mu_n \in W^{1,sp}\cap W^{s-1, \frac{sp}{s-1}} \mbox{ uniformly in }n.
\end{equation}
As a consequence, 
\begin{equation}\label{eqUniformControlDerivativesHN}
\partial h,\partial h_n\in L^r  \mbox{ uniformly  in } n   \mbox{ for every given } \frac1r > \frac{1}{sp} + \frac{1}{p_\kappa}
\end{equation}
 by \cite[Proposition 4]{ClopFaracoMateuOrobitgZhong} and \cite[Corolario 3.8]{BaisonThesis}, and, as we have proven above,
\begin{equation}\label{eqUniformControlFractionalDerivativesHN}
D^{s-1} h,D^{s-1} h_n\in L^r  \mbox{ uniformly  in } n   \mbox{ for every given } \frac1r > \frac{s-1}{sp} + \frac{1}{p_\kappa}.
\end{equation}

Thus, we only need to deal with the homogeneous norm, that is, $\norm{D^{s-1}\partial h}_{L^q}$.
From \rf{eqBeltrami} we deduce that
\begin{align*}
\partial h_n 
	& = \mu_n \partial \Beurling  h_n + \partial\mu_n \Beurling h_n +  \partial \mu_n 
\end{align*}
and differentiating we get
\begin{align*}
D^{s-1} \partial h_n 
	& = \mu_n   D^{s-1} \Beurling  \partial h_n - [\mu_n, D^{s-1}](\Beurling  \partial h_n) +  \partial\mu_n D^{s-1} \Beurling h_n - [\partial \mu_n, D^{s-1}](\Beurling h_n)  +   D^{s-1} \partial \mu_n
\end{align*}
leading to
\begin{align}\label{eqOperatorReadyToInvert}
 \norm{ (I- \mu_n  \Beurling )(D^{s-1} \partial h_n)  }_{L^q}
\nonumber	  & \leq \norm{[\mu_n, D^{s-1}](\Beurling  \partial h_n) }_{L^q}+  \norm{\partial\mu_n D^{s-1} \Beurling h_n}_{L^q} +\norm{[\partial \mu_n, D^{s-1}](\Beurling h_n) }_{L^q}\\
	&	\quad +\norm{  D^{s-1} \partial \mu_n }_{L^q}  = \circled{1}+\circled{2}+\circled{3}+\circled{4}.
\end{align}
Note that in the previous case, \rf{eqIdentityBeforeArmaggedon} had a simpler form, but the essential ideas to control the norm of the right-hand side are the same. We will find $\frac1{p_\kappa}<\frac1q<\frac{1}{p_\kappa'}$ such that $\circled{j} \leq C_{\kappa,q}$ for $j\in\{1,2,3,4\}$.

First of all, by Theorem \ref{theoHofmann} we have that
$$\circled{1}=\norm{[\mu_n, D^{s-1}](\Beurling  \partial h_n) }_{L^q}\leq \norm{ D^{s-1}\mu_n}_{L^{\frac{sp}{s-1}}}\norm{\Beurling  \partial h_n}_{L^{r_1}}$$
for $\frac{1}{q}=\frac{s-1}{sp}+\frac{1}{r_1}$. The first term is uniformly bounded by \rf{eqUniformControlDerivativesMuN}, and the last one is controlled by \rf{eqUniformControlDerivativesHN} as long as  $\frac{1}{r_1}>\frac{1}{sp} + \frac{1}{p_\kappa}$.  It is possible to find such a value for $r_2$ as long as 
$$\frac{1}{q}>\frac{s-1}{sp} +\frac{1}{sp} + \frac{1}{p_\kappa} = \frac{1}{p}+\frac{1}{p_\kappa}.$$

On the other hand, using the commutation of fractional derivatives and the Beurling transform and H\"older's inequality, we have that 
$$\circled{2}=\norm{\partial\mu_n  \Beurling D^{s-1} h_n}_{L^q} \leq \norm{\partial\mu_n}_{L^{sp}}\norm{ \Beurling D^{s-1} h_n}_{L^{r_2}}$$
for $\frac{1}{q}=\frac{1}{sp}+\frac{1}{r_2}$. The first term is uniformly bounded by \rf{eqUniformControlDerivativesMuN}, and the last one is controlled by \rf{eqUniformControlFractionalDerivativesHN} as long as  $\frac{1}{r_2}>\frac{s-1}{sp} + \frac{1}{p_\kappa}$.  It is possible to find such a value for $r_2$ as long as 
$$\frac{1}{q}>\frac{1}{sp} +\frac{s-1}{sp} + \frac{1}{p_\kappa} = \frac{1}{p}+\frac{1}{p_\kappa}.$$

The latter two terms are bounded as before. By Theorem \ref{theoHofmann} we have that
$$\circled{3}=\norm{[\partial \mu_n, D^{s-1}](\Beurling h_n)}_{L^q}\leq \norm{D^{s-1}\partial\mu_n}_{L^p}\norm{\Beurling h_n}_{L^{r_3}} $$
for $\frac{1}{q}=\frac{1}{p}+\frac{1}{r_3}$, the last term being uniformly controlled for $\frac{1}{r_3}>\frac{1}{p_\kappa}$ by Theorem \ref{theoAIS}. It is possible to find such a value for $r_3$ as long as 
$$\frac{1}{q}>\frac{1}{p}+\frac{1}{p_\kappa}.$$
Finally,
$$\circled{4}=\norm{D^{s-1} \partial \mu_n}_{L^q} \leq C_{\kappa,q} \mbox{ for every }\frac{1}{q} \geq \frac{1}{p}.$$
This facts, together with \rf{eqOperatorReadyToInvert} and Theorem \ref{theoAIS}, show that \rf{eqUniformControlInH} holds for $s>1$ with exactly the same restrictions as when $s<1$, that is, when $\frac{1}{p}+\frac{1}{p_\kappa}<\frac{1}{q}<\frac{1}{p_\kappa'}$. The theorem follows by the Banach-Alaoglu Theorem again.
\end{proof}

\renewcommand{\abstractname}{Acknowledgements}
\begin{abstract}
The author wants to express his gratitude towards Xavier Tolsa, Eero Saksman and Daniel Faraco for their encouragement in solving this problem, Albert Clop and Antonio Luis Bais\'on for their advice on the matter.  

The author was funded by the European Research Council under the grant agreement 307179-GFTIPFD, and partially funded by AGAUR - Generalitat de Catalunya (2014 SGR 75). 
\end{abstract}

\bibliography{../../../bibtex/Llibres}
\end{document}